\documentclass[12]{amsart}
\usepackage{amsmath,amsthm,amsxtra}
\usepackage{bm}
\usepackage{graphics}
\usepackage{graphicx}	
\usepackage{float}
\usepackage{thmtools}
\usepackage{amssymb,mathrsfs}
\usepackage{enumerate}
\usepackage{mathtools}
\usepackage{epsfig}
\usepackage{tikz-cd}
\usetikzlibrary{calc,backgrounds}
\usetikzlibrary{fit,calc,positioning,decorations.pathreplacing,matrix}
\usepackage{pstricks-add}
\usepackage{pst-node,pst-plot}
\usepackage{latexsym}
\usepackage{amsfonts}
\usepackage[colorlinks=true]{hyperref}
\usepackage{xfrac}
\usepackage[capitalize]{cleveref}
\usepackage{enumitem}

\usepackage{tikz}
\usetikzlibrary{automata,positioning,matrix}

\newtheorem{theorem}{Theorem}[section]
\newtheorem*{theorem*}{Theorem}
\newtheorem{corollary}[theorem]{Corollary}

\newtheorem{lemma}[theorem]{Lemma}
\newtheorem{question}[theorem]{Question}
\newtheorem{proposition}[theorem]{Proposition}

\theoremstyle{definition}
\newtheorem{definition}[theorem]{Definition}
\newtheorem{remark}[theorem]{Remark}

\theoremstyle{definition}
\newtheorem{example}[theorem]{Example}

\newcommand{\Z}{\mathbb{Z}}
\newcommand{\Q}{\mathbb{Q}}

\hypersetup{citecolor={blue}}

\def\R{\mathbb{R}}

\DeclareMathOperator{\conv}{conv}

\DeclareMathOperator{\adj}{adj}
\DeclareMathOperator{\rad}{rad}
\DeclareMathOperator{\sspan}{span}

\def\NN{\mathbb{N}}

\def\A{\mathcal{A}}
\def\B{\mathcal{B}}

\def\NN{\mathbb{N}}

\def\cM{\mathcal{M}}
\def\Q{\mathbb{Q}}

\def\ie{{i.e.}}  
\def\GL{\textrm{GL}(d, \mathbb{Z})}
\def\GLtwo{\textrm{GL}(2, \mathbb{Z})}
\def\Id{\textrm{Id}}

\DeclareMathOperator{\Aut}{Aut}

\DeclareMathOperator{\supp}{supp}

\DeclareMathOperator{\End}{End}

\DeclareMathOperator{\Homeo}{Homeo}

\newcounter{sigmavariable}
\newcounter{Lvariable}
\DeclareMathOperator{\trace}{trace}
\DeclareMathOperator{\Cent}{Cent}

\usepackage{scalerel,stackengine}
\stackMath
\newcommand\reallywidehat[1]{%
\savestack{\tmpbox}{\stretchto{%
  \scaleto{%
    \scalerel*[\widthof{\ensuremath{#1}}]{\kern-.6pt\bigwedge\kern-.6pt}%
    {\rule[-\textheight/2]{1ex}{\textheight}}
  }{\textheight}%
}{0.5ex}}%
\stackon[1pt]{#1}{\tmpbox}%
}

\definecolor{zzttqq}{rgb}{0.6,0.2,0.}

\title[Normalizer of odometers and substitutive subshifts]{Large normalizers of $\Z^{d}$-odometer systems and realization on substitutive subshifts}

\author{Christopher Cabezas} 
\address{Department of Mathematics, Université de Liège, All\'ee de la découverte 12 (B37), 4000 Liège, Belgium}
\email{ccabezas@uliege.be}

\author{Samuel Petite} 
\address{LAMFA,  CNRS, UMR 7352, Universit\'e de Picardie Jules Verne, 80 000 Amiens, France}
\email{samuel.petite@u-picardie.fr}

\thanks{The authors acknowledge the financial support of  ANR project IZES ANR-22-CE40-0011 and ECOS projet ACEDic C21E04.}


\begin{document}

\begin{abstract}
For a $\Z^{d}$-topological dynamical system $(X, T, \Z^d)$, an \emph{isomomorphism} is a self-homeomorphism  $\phi : X\to X$ such that for some matrix $M\in GL(d,\Z)$ and any ${\bm n}\in \Z^{d}$, $\phi\circ T^{{\bm n}}=T^{M{\bm n}}\circ \phi$, where $T^{\bm n}$  denote the self-homeomorphism of $X$ given by the action of ${\bm n}\in {\mathbb Z}^d$. The collection of all the isomorphisms forms a group that is the normalizer of the set of transformations $T^{\bm n}$.
In the one-dimensional case, isomorphisms correspond to the notion of \emph{flip conjugacy} of dynamical systems and by this fact are also called \emph{reversing symmetries}.

Isomorphisms are not well understood even for classical systems. We present a description of them for odometers and more precisely for constant-base $\Z^{2}$-odometers, which is surprisingly not simple. We deduce a complete description of the isomorphisms of some minimal $\Z^{d}$-substitutive subshifts. This enables us to provide the first known example of a minimal zero-entropy subshift with the largest possible normalizer group. 
\end{abstract}
	\maketitle
	
	\section{Introduction}
	
This article concerns $\Z^d$-odometers  and their  symbolic extensions. More specifically, for such dynamical system $(X, T, \Z^d)$, we study their \emph{isomorphisms}, which are self-homeomorphisms $\phi : X\to X$ such that for some linear transformation $M\in GL(d,\Z)$ and any ${\bm n}\in \Z^{d}$, $\phi\circ T^{{\bm n}}=T^{M{\bm n}}\circ \phi$, where $T^{\bm n}$ denotes the self-homeomorphism of $X$ given by the $\Z^d$-action $T$.

Isomorphisms associated with the linear transformation $M= \Id_{\R^{d}}$ are, in essence, nothing more than self-conjugacies of the system. They are commonly referred to as \emph{automorphisms} of dynamical systems. Therefore, an isomorphism can be thought of as a self-conjugacy up to a $\GL$-transformation (see for example \cite{baake2018reversing,cabezas2023homomorphisms}).
In the one-dimensional case ($d=1$), isomorphisms correspond to the notion of \emph{flip conjugacy} in dynamical systems \cite{bezuglyi2008fullgroups}. Because of this, they are also known as \emph{reversing symmetries} (see \cite{goodson1999inverse, baake2006structure}). 
The automorphisms can be algebraically defined as elements of the centralizer of the group $\left\langle T \right \rangle$, considered as a subgroup of all self-homeomorphisms $\Homeo(X)$ of $X$. From this algebraic perspective, isomorphisms can be seen as elements of the normalizer group of $\left\langle T \right \rangle$ within the group $\Homeo(X)$. Consequently, the automorphism group is a normal subgroup of the normalizer. The quotient is isomorphic to a linear subgroup of $\GL$, called the \emph{linear representation  group}, achieved through the identification of an isomorphism $\phi$ with its matrix $M$. The automorphism group is never trivial (it always contains the transformations $T^{\bm n}$, $\bm n \in \Z^d$), but generally, the existence of an isomorphism for a given $M\in GL(d,\Z)$ remains an open problem, which is significant in the context of higher-rank actions. 

There are few general results on these groups so natural questions on the algebraic properties of the group of isomorphisms (is it finite? is it amenable? what are the subgroups? etc.), on their relations with the dynamics, or on a description of their actions, are wide open. In this article, we explore these issues using classic and widely studied examples of odometers and substitutive Toeplitz subshifts.

Even though the dynamics of an odometer is one of the best understood, a partial description of their isomorphisms has only been initiated in \cite{giordano2019zdodometer} for $d=2$ and pursued for higher rank in \cite{merenkov2022odometers,sabitova2022,sabitova2022number}. These works were essentially motivated by the fact that isomorphisms define specific transformations of  orbit equivalence. We complete these studies, but  mainly for dimension $d=2$, by describing explicitly the group structure of their  normalizers.  This leads to a complete classification  of them (Theorem \ref{GeneralTheoremDifferentCasesNormalizer}) for constant-base odometer systems, i.e., whose base is given by the iteration of the same expansive matrix. The classification depends on arithmetical properties of the expansive matrix. In addition, we provide computable arithmetic conditions for deciding whether a given matrix arises as an isomorphism of a constant-base odometer. Our techniques use elementary arithmetic, but extending our results to higher ranks would require sophisticated notions of number theory.

Interest in the  study of the normalizer group for multidimensional subshifts has  increased in recent  years. These objects represent geometrical and combinatorial symmetries, such as rotations and reflections, that a particular subshift possesses. Similar studies were initiated in the framework of tilings and Delone sets (see e.g., \cite{Robinson1996}), where the existence of some point set symmetries is at the heart of the discovery of aperiodicity of quasicrystals \cite{shechtman1984metallic}. For instance, the Penrose tiling serves as a model with ten-fold symmetry \cite{Penrose1979}. However, still few results and no systematic description of the analogous normalizer group for $\R^d$-flows exist. For a self-similar tiling, its linear representation group  is known to be countable and isomorphic to its mapping class group \cite{Kwapisz2011}.   
Recently, classical examples of multidimensional subshifts have been considered to study their normalizers \cite{baake2021number,baake2018reversing,  bustos2021admissible,cabezas2023homomorphisms}. These works have yielded sufficient conditions (of a combinatorial nature in \cite{bustos2021admissible} and  a geometrical nature in \cite{cabezas2023homomorphisms}) on a substitutive subshift  so that the quotient of the normalizer by the automorphism group is finite, and each group is virtually $\Z^d$. 
The article \cite{baake2021number} concerns different classes of non-minimal subshifts with positive entropy having a normalizer not virtually isomorphic to its automorphism group.

However, the question remains open whether an infinite linear representation group is also possible for minimal, deterministic (zero-entropy) subshifts. In this article, we provide a positive answer to this problem by completely describing the normalizer group of a family of minimal subshifts (Theorem \ref{prop:MainresultSection4}). 
The examples are substitutive  Toeplitz subshifts that can be presented in an effective  way, in the algorithmic sense. These subshifts are highly deterministic meaning they have zero entropy. More precisely, the rate of growth of their complexity is polynomial of degree equal to the dimension $d$.  In particular, we exhibit examples  where the linear representation group is the largest one, i.e., equal to $\GL$. Together with the  results of \cite{baake2021number}, this shows that the normalizer is not restricted by the complexity. More surprisingly, low complexity is not enough to ensure the amenability of the normalizer group in dimension $d>1$. This phenomenon differs from what  happens in dimension  one  where the amenability of  the automorphism  group (of finite-index in the normalizer group) is shown for a large class of zero-entropy subshifts (see \cite{CovenQuasYassawi, CyrKra2015, donoso2016lowcomplexity} for linear complexity subshifts and  \cite{CyrKra2016, CyrKra2020} for subexpontential complexity  subshifts) and any Toeplitz subshift \cite{DonosoToeplitz2017}.

This article is organized as follows: Section \ref{sec:Basics} is devoted to the background on topological and symbolic dynamics. In particular,  relations between the normalizer of a system with the one of its maximal equicontinuous factor are exhibited. The next section mostly concerned with the study of  the normalizers of odometers in order to describe them (Theorem \ref{GeneralTheoremDifferentCasesNormalizer}). Finally, we construct examples of  multidimensional subshifts with an odometer as their maximal equicontinuous factors in Section \ref{sec:NormalizerSubshiftEx}. We characterize their infinite linear representation groups (Theorem \ref{prop:MainresultSection4}) by using our study of normalizers of odometers and the relations with their extensions.

	\section{Definitions and basic properties}\label{sec:Basics}
	
	\subsection{Topological dynamical systems}\label{SectionTopologicalDynamicalSystem}

We recall that a \emph{topological dynamical system}  $(X, T, \Z^d)$ is given by a continuous left-action $ T \colon \Z^d \times X \to X$  on a compact metric space $X$ equipped with a distance. This provides a family of  self-homeomorphisms  of $X$:   $\{ T^{\bm n}\colon {\bm n} \in \Z^d\}$, also denoted by $\langle T \rangle $, such that $T^{\bm m}\circ T^{\bm n} = T^{\bm m+\bm n}$ for any ${\bm m,\bm n} \in \Z^d$. In particular, the homeomorphisms $T^{\bm n}$ commute with each other.  
The  \emph{orbit} of a point $x \in X$  is the set $\mathcal{O}(x,T)=\{T^{\bm n}(x)\colon {\bm n}\in \Z^d\}$. 
We will be mainly concerned by topological dynamical systems that are {\em minimal}, i.e., where any orbit is dense.  For minimal systems, there is no topological way to separate the orbits. 

An important type of topological dynamical systems are the equicontinuous ones. A topological dynamical system $(X,T,\Z^{d})$ is said to be \emph{equicontinuous} if the set of maps $\{T^{{\bm n}}\colon {\bm n}\in \Z^{d}\}$ forms an equicontinuous family of homeomorphisms. The equicontinuous systems are, in some sense, the simplest dynamical systems to describe.

In the following, we define the endomorphisms and isomorphisms of a  topological dynamical system, which are the central objects of study of this article. Isomorphisms represent internal symmetries of a given system that do not commute with the action, such as rotations and reflections. We  follow the notations of 	\cite{cabezas2023homomorphisms, LabbeRao2021, Labbe2021}.

	\begin{definition}
Let $(X,T,\Z^{d})$, $(Y,S,\Z^{d})$ be two topological dynamical systems and ${M\in \GL}$. 
An  $M$-\emph{epimorphism}  is a continuous surjective map ${\phi: X\to Y}$ that for any ${\bm n}\in \Z^{d}$ satisfies $\phi\circ T^{{\bm n}}= T^{M{\bm n}}\circ \phi$. 
When $(X,T,\Z^{d})=(Y,S,\Z^{d})$, it is called an $M$-\emph{endomorphism}. Moreover, if $\phi$ is invertible, it is called an $M$-\emph{isomorphism}.

We simply call a $\GL$-\emph{endomorphism} (or $\GL$-\emph{isomorphism}), any $M$-endomorphism (resp. isomorphism) for some $M \in \GL$. 	
	\end{definition}

In other terms, the $\GL$-isomorphisms are conjugacies of $\Z^{d}$-actions, up to a $\GL$-transformation, i.e., an orbit equivalence with a constant orbit cocycle. 
In the special case where $M$ is the identity matrix $\Id_{\R^{d}}$, the $\Id_{\R^{d}}$-endomorphisms and $\Id_{\R^{d}}$-isomorphisms are called \emph{endomorphisms} (or \emph{self-factor maps}) and \emph{automorphisms} (or \emph{self-conjugacies}), respectively.

When the action is \emph{aperiodic}, \ie, the stabilizer of any point is trivial or $T^{\bm n} x= x $ only for ${\bm n} ={\bm 0}$, the matrix $M$ associated with a $\GL$-endomorphism $\Phi$ is unique: $\Phi$ is an $M_1$- and $M_2$-endomorphism if and only if $M_1=M_2$. 
In the following, we fix different notations that we will use throughout this article:
\begin{itemize}
		\item $N_{M}(X,T)$, with $M \in \GL$, denotes the set of all $M$-endomorphisms of $(X,T)$; 
		
		\item $N(X,T)$ denotes the set of all the $\GL$-endomorphisms of the dynamical system $(X,T, \Z^d)$, \ie 
		$$N(X,T)=\bigcup\limits_{M\in \GL}N_M(X,T).$$
	  The set of $\GL$-isomorphisms is denoted by $N^*(X,T)$, i.e.,
	  $$N^*(X,T) = \{\phi \in N(X,T)\mid \phi\ \text{is a}\ \GL\text{-isomorphism}\}.$$
	  
	  In algebraic terms, this set is a group and, when the action is aperiodic, corresponds to the normalizer of the group action $\left\langle T\right\rangle$ in the group of self-homeomorphisms $\Homeo(X)$ of $X$, that is, the set of self-homeomorphisms $\phi \colon X \to X$ such that $\phi \circ T^{\bm n} \circ \phi^{-1} \in \{T^{\bm m}: {\bm m} \in \Z^d\}$ for any ${\bm n}\in \Z^d$;   
		
		\item $\End(X,T)$ and $\Aut(X,T)$ denotes respectively the set of all endomorphisms and automorphisms of $(X, T, \Z^d)$.
		
		\item We define the \emph{linear representation semigroup} $\vec{N}(X,T)$ as
		$$\vec{N}(X,T)=\{M\in \GL\mid N_{M}(X,T)\neq \emptyset\}.$$
\end{itemize}

Notice that for $\phi\in N_{M_{1}}(X,T )$ and $\psi\in N_{M_{2}}(X,T)$, the composition $\phi\circ \psi$ belongs to $N_{M_{1}M_{2}}(X,T)$. Moreover, if $\phi$ is an $M$-isomorphism, then its inverse is an $M^{-1}$-isomorphism, so the sets $N_{M}(X,T)$ are not semigroups (except when $M$ is the identity matrix).		

Concerning the linear representation semigroup $\vec{N}(X,T)$, it is direct to check that the isomorphism class of $\vec{N}(X,T)$ is invariant under conjugacy.
However, it is not necessarily a group, since the existence of an $M$-endomorphism associated with a matrix $M\in \GL$ does not necessarily imply the existence of an $M^{-1}$-endomorphism. 
Nevertheless, we have the following result when a dynamical system is coalescent. Recall that a topological dynamical system $(X,T,\Z^{d})$ is \emph{coalescent} when any endomorphism of $(X,T,\Z^{d})$ is invertible \cite[Chapter 5]{auslander1988minimal}.
	
\begin{proposition}\label{CoalescenceOfHommorphismsForCoalescentSystems}
		Let $(X,T,\Z^{d})$ be a coalescent system. If the linear representation semigroup $\vec{N}(X,T)$ is a group, then any $\GL$-endomorphism in $N(X,T)$ is invertible, \ie, $N(X,T) = N^{*}(X,T)$. 
	\end{proposition}

Equicontinuous systems are examples of coalescent systems \cite[Chapter 5]{auslander1988minimal}. 
 	
\begin{proof}
		Let $\phi,\psi$ be respectively an $M$- and $M^{-1}$-endomorphism onto $(X,T,\Z^{d})$. Then $\phi\circ\psi$ is an endomorphism onto $(X,T,\Z^{d})$. Since $(X,T,\Z^{d})$ is coalescent, then $\phi\circ \psi$ is invertible. Hence $\psi$ is injective, and thus invertible. It follows then that $\phi$ is also invertible.
	\end{proof}
	
A particular case is when the linear representation semigroup of a coalescent system is finite. In this case, as a finite semigroup,  it is a group. So any $\GL$-endomorphism is invertible.

The groups $\left\langle T\right\rangle$ and $\Aut(X,T)$ are normal subgroups of $N^{*}(X,T)$ (the group of isomorphisms). 
More precisely, for aperiodic systems the following exact sequence holds:
	\begin{align}\label{ExactSequenceForNormalizer}
		1 & \to &     \Aut(X,T) \quad & \to \quad  N^{*}(X,T) & \xrightarrow{j}  \quad  \vec{N^{*}}(X,T)\quad   \to  & \quad  1,
	\end{align}
where all the maps are the canonical injections and $j(\phi)= M$ whenever $\phi \in N_M(X,T)$.  

A \emph{factor map} $\pi: (X,T,\Z^{d})\to (Y,S,\Z^{d})$ between two topological dynamical systems $(X, T, \Z^d)$ and $(Y, S,\Z^{d})$ is a continuous onto map commuting with the action, \ie, $\pi \circ T^{\bm n}=S^{\bm n} \circ \pi $ for any ${\bm n} \in \Z^d$. The system $(Y, S,\Z^{d})$ is said to be a \emph{factor} of $(X, T, \Z^d)$, and $(X, T, \Z^d)$ is an \emph{extension} of $(Y, S,\Z^{d})$.
If $|\pi^{-1}(\{y\})|\leq K<\infty$ for all $y\in Y$, we say that $\pi$ is \emph{finite-to-1}.
Sometimes there exists a $G_{\delta}$-dense subset $Y_{0}\subseteq Y$ where any $y\in Y_{0}$ satisfies $|\pi^{-1}(\{y\})|=1$. In this case, the factor map $\pi$ is said to be \emph{almost} \emph{1-to-1}. We recall that when the system $(Y,S, \Z^d)$ is minimal, the existence of one point with only one preimage is enough to ensure that the map is almost 1-to-1. 

For every topological dynamical system, there exists at least one equicontinuous factor, which is the system given by one point. Furthermore, any topological dynamical system admits a \emph{maximal equicontinuous factor}, \ie, a factor $\pi_{eq}:(X,T,\Z^{d})\to(X_{eq},T_{eq},\Z^{d})$ where $(X_{eq},T_{eq},\Z^{d})$ is an equicontinuous system, satisfying the following universal property: for any other equicontinuous factor $\pi:(X,T,\Z^{d})\to (Y,S,\Z^{d})$ there exists a factor map $\phi:(X_{eq},T_{eq},\Z^{d})\to (Y,S,\Z^{d})$ such that $\pi=\phi\circ \pi_{eq}$ \cite[Theorem 1, Chapter 9]{auslander1988minimal}. 
Also, in the particular case where $\pi:(X,T,\Z^{d})\to (Y,S,\Z^{d})$ is an almost 1-to-1 factor on an equicontinuous system $(Y,S,\Z^{d})$, this factor is the maximal equicontinuous factor of $(X,T,\Z^{d})$. As typical examples of this case, are the odometer systems (see next section) that are almost 1-to-1 factors of particular symbolic systems \cite{cortez2006toeplitz,cortez2008godometers,downarowicz2005survey, AbdalaouiLemanczyckdelaRue}.
	
A factor map $\pi:(X,T,\Z^{d})\to(Y,S,\Z^{d})$ is \emph{compatible} if any endomorphism $\phi\in \End(X,T)$ preserves the $\pi$-fibers, \ie, $\pi(\phi(x))= \pi(\phi(y))$ for any $x,y\in X$, such that $\pi(x)= \pi(y)$.
In the same spirit, we say that a factor $\pi$ is \emph{compatible with $\GL$-endomorphisms} if for any $\GL$-endomorphism $\phi\in N(X,T)$, $\pi(\phi(x))= \pi(\phi(y))$ for any $x,y\in X$, such that $\pi(x)= \pi(y)$.

The compatibility property allow us to relate $\GL$-endomorphisms between factor systems. The next lemma follows the ideas from \cite[Theorem 3.3]{CovenQuasYassawi} but in a larger context.

	\begin{lemma}\label{CompatibilityPropertyFactors}
		Let $(X,T,\Z^{d})$, $(Y,S,\Z^{d})$ be two minimal systems, such that $\pi:(X,T,\Z^{d})\to (Y,S,\Z^{d})$ is a compatible factor. Then, there is a semigroup homomorphism $\hat{\pi}:\End(X,T)\to\End(Y,S)$ such that
		
		\begin{enumerate}[label=\arabic*.,ref=\arabic*.]
			\item\label{CompatibilityPropertyEnumerate1} $\hat{\pi}(\phi)(\pi(x))=\pi(\phi(x))$ for all $\phi\in \End(X,T)$ and $x\in X$.
			\item\label{CompatibilityPropertyEnumerate2} $\hat{\pi}(\Aut(X,T))\leqslant  \Aut(Y,S)$.
			\item\label{CompatibilityPropertyEnumerate3} For all $\psi\in \End(Y,S)$, $|\hat{\pi}^{-1}(\{\psi\})|\leq \min\limits_{y\in Y}|\pi^{-1}(y)|$.
		\end{enumerate}
		
		Moreover, if $\pi$ is compatible with $\GL$-endomorphisms, there is an extension of $\hat{\pi}:N(X,T)\to N(Y,S)$ defined as in \cref{CompatibilityPropertyEnumerate1} for all $\phi\in N(X,T)$, satisfying the following properties
	\begin{enumerate}[label=\roman*),ref=\roman*]
	    \item\label{CompatibilityPropertyEnumeratei} $\hat{\pi}(N_{M}(X,T))\subseteq N_{M}(Y,S)$, for any $M\in GL(d,\Z)$.
	    \item\label{CompatibilityPropertyEnumerateii} For any $M \in \GL$, the map $\hat{\pi}:N_{M}(X,T)\to N_{M}(Y,S)$ is at most $\min\limits_{y\in Y}|\pi^{-1}(\{y\})|$-to-1.
	    \item\label{CompatibilityPropertyEnumerateiii} For any $\phi\in  \hat{\pi}^{-1} (N^*(Y,S))$, the cardinality of the $\pi$-fiber is non decreasing under the $\GL$-isomorphism $\hat{\pi}(\phi)^{-1}$. In other terms, for any integer $n \geq 1$, the map $\hat{\pi}(\phi)$ satisfies
$$\{y\in Y\colon |\pi^{-1}(\{y\})|\ge n\}\subset \hat{\pi}(\phi)\left(\{y\in Y\colon |\pi^{-1}(\{y\})|\ge n\}\right).$$
	\end{enumerate}
\end{lemma}
	
	\begin{proof}
		Set $\phi\in \End(X,T)$. By definition, the map $\hat{\pi}(\phi):Y\to Y$ given by $\hat{\pi}(\phi)(\pi(x))=\pi(\phi(x))$ is well defined and is surjective by minimality of $(Y,S,\Z^{d})$. So $\hat{\pi}(\phi)$ is an endomorphism of $(Y,S,\Z^{d})$. Moreover, if $\phi$ is an automorphism of $(X,T,\Z^{d})$, then $\hat{\pi}(\phi)$ is invertible. Indeed, $\hat{\pi}(\phi)\circ \hat{\pi}(\phi^{-1})\circ \pi=\pi\circ\phi\circ \phi^{-1}=\pi$, so we conclude that $\hat{\pi}(\phi)\circ \hat{\pi}(\phi^{-1})=\Id_{Y}$.
		
To prove \cref{CompatibilityPropertyEnumerate3}, fix any $\psi\in \End(Y,S)$ and suppose that $\min\limits_{y\in Y}|\pi^{-1}(\{y\})|=c<\infty$ (if not, then there is nothing to prove). Let $x_{0}\in X$ and $y_{0}\in Y$ be such that $|\pi^{-1}(\{y_{0}\})|=c$, and $y_{0}=\psi(\pi(x_{0}))$. Assume there exists $c+1$ endomorphisms $\phi_{0},\ldots,\phi_{c}$ of $(X,T,\Z^{d})$, in $\hat{\pi}(\{\psi\})^{-1}$. The compatibility then implies that
		 $y_{0}=\psi(\pi(x_{0}))=\pi(\phi_{0}(x_{0}))=\cdots=\pi(\phi_{c}(x_{0}))$. So, by the pigeonhole principle, there must exist two different indices $0\leq i,j\leq c$ such that $\phi_{i}(x_{0})=\phi_{j}(x_{0})$. The minimality of $(X,T,\Z^{d})$ then gives that $\phi_{i}=\phi_{j}$.
		
The proofs concerning the items \ref{CompatibilityPropertyEnumeratei} and \ref{CompatibilityPropertyEnumerateii} on $\GL$-endomorphisms use similar arguments and are left to the reader. 

To prove \cref{CompatibilityPropertyEnumerateiii}, we consider any $y \in Y$ with $n$ preimages $x_1, \ldots, x_n \in X$. Since $\phi$ is onto, there are $x'_1, \ldots, x'_n \in X$ such that $\phi(x'_i)= x_i$. Notice that, $\hat{\pi}(\phi) (\pi(x'_i)) = \pi(\phi(x'_i))= \pi(x_i) =y$. It follows that $ \hat{\pi}(\phi) (\pi(x'_i)) = \hat{\pi}(\phi) (\pi(x'_j))$ for any indices $i,j = 1, \ldots, n$.
Since $\hat{\pi}(\phi)$ is invertible, all the element $x'_i$ belongs to the same $\pi$-fiber of $z \in Y$. Thus $z$ admits at least $n$ preimages by $\pi$ and satisfies $\hat{\pi}(\phi)(z) = y$. The claim follows.  
	\end{proof}
	
It is known that factor maps between equicontinuous systems are compatible \cite{auslander1988minimal}, but as we will see in the next section, they are not necessarily compatible with $\GL$-endomorphisms (see \cref{rem:compatibilityCounterExample}). Nevertheless, the maximal equicontinuous factor is an example of a factor compatible with $\GL$-endomorphisms as proved in \cite[Theorem 5 and Corollary 3]{baake2018reversing}. This can also be seen by using the universal property of the maximal equicontinuous factor.

\begin{lemma}\cite[Corollary 3]{baake2018reversing}\label{MaximalEquicontinuousFactorCompatibleWithHomomorphisms} Let $(X,T,\Z^{d})$ be a minimal topological dynamical system and let ${\pi_{eq}:(X,T,\Z^{d})\to (X_{eq},T_{eq},\Z^{d})}$ denote its maximal equicontinuous factor. 
Then $\pi_{eq}$ is compatible with $\GL$-endomorphisms.
\end{lemma}

 \cref{CompatibilityPropertyFactors} and \cref{MaximalEquicontinuousFactorCompatibleWithHomomorphisms} illustrate that to study $\GL$-endomorphisms, a first step is to understand the equicontinuous systems. This will be done in \cref{SectionOdoemterSystems} for the class of minimal equicontinous Cantor systems: the odometer systems.

\subsection{$\Z^{d}$-Odometer systems}\label{SectionOdoemterSystems}
	
Odometer systems are the equicontinuous minimal Cantor systems. They are therefore the maximal equicontinuous factor for a large family of minimal symbolic systems, such as Toeplitz flows and some substitutive subshifts. We refer to \cite{cortez2006toeplitz} for the study of odometer systems and $\Z^d$-Toeplitz sequences.  We use  the same notations.

In this section, we briefly recall the basic definitions of odometers. Subsequently, we describe the linear representation semigroup of odometer systems (\cref{LemmaNessesaryConditionNormalizerOdometer}), which we then use to completely characterize it for constant-base $\Z^{2}$- odometer systems (\cref{GeneralTheoremDifferentCasesNormalizer}). The initial exploration of $\GLtwo$-endomorphisms between $\Z^{2}$-odometer systems was initiated in \cite{giordano2019zdodometer} and later extended to higher dimensions in \cite{merenkov2022odometers}.

\subsubsection{Basic definitions}

Let $Z_{0}\geqslant Z_{1}\geqslant \ldots \geqslant Z_{n}\geqslant Z_{n+1}\geqslant\ldots$ be a nested sequence of finite-index subgroups of $\Z^{d}$ such that $\bigcap\limits_{n\geq 0}Z_{n}=\{{\bm 0}\}$ and let $\alpha_{n}:\Z^{d}/Z_{n+1}\to \Z^{d}/Z_{n}$ be the function induced by the inclusion map. We consider the inverse limit of these groups
	$$\overleftarrow{\Z^{d}}_{(Z_{n})}=\lim\limits_{\leftarrow n} (\Z^{d}/Z_{n},\alpha_{n}),$$
\ie, $\overleftarrow{\Z^{d}}_{(Z_{n})}$ is the subset of the product $\prod\limits_{n\geq 0} \Z^{d}/Z_{n}$ consisting of the elements ${\overleftarrow{g}=({\bm g}_{n})_{n\geq 0}}$ such that $\alpha_{n}({\bm g}_{n+1})={\bm g}_{n}\ (\text{mod}\ Z_{n})$ for all $n\geq 0$. The odometer $\overleftarrow{\Z^{d}}_{(Z_{n})}$ is a compact 0-dimensional topological group, whose topology is generated by the cylinder sets
	$$[{\bm a}]_{n}=\left\{\overleftarrow{g}\in \overleftarrow{\Z^{d}}_{(Z_{n})}: {{\bm g}}_{n}={{\bm a}}\right\},$$
	
	\noindent with ${{\bm a}}\in \Z^{d}/Z_{n}$, and $n\geq 0$. Now, consider the group homomorphism $\kappa_{(Z_{n})}:\Z^{d}\to \prod\limits_{n\geq 0} \Z^{d}/Z_{n}$ defined for ${\bm n}\in \Z^{d}$ as
	$$\kappa_{(Z_{n})}({\bm n})=[{\bm n}\ \text{mod}\ Z_{n}]_{n\geq 0}.$$
	
The image of $\Z^{d}$ under $\kappa_{(Z_{n})}$ is dense in $\overleftarrow{\Z^{d}}_{(Z_{n})}$, allowing us to define the $\Z^{d}$-action ${\bm n}{\bm +}\overleftarrow{g}=\kappa_{(Z_{n})}({\bm n})+\overleftarrow{g}$, where ${\bm n}\in \Z^{d}$ and $\overleftarrow{g}\in \overleftarrow{\Z^{d}}_{(Z_{n})}$. This action is well-defined and continuous. The resulting topological dynamical system $(\overleftarrow{\Z^{d}}_{(Z_{n})}, {\bm +},\Z^{d})$ is called a $\Z^{d}$-\emph{odometer system}, and for the rest of this article, we will simply use the notation $\overleftarrow{\Z^{d}}_{(Z{n})}$, i.e., denoted just by its phase space.
Similarly, its set of automorphisms, endomorphisms, and linear representation semigroup will be denoted as $\Aut(\overleftarrow{\Z^{d}}_{(Z_{n})})$, $N(\overleftarrow{\Z^{d}}_{(Z_{n})})$, and $\vec{N}(\overleftarrow{\Z^{d}}_{(Z_{n})})$, respectively.

Notice that the ``return times'' of the action to a cylinder set $[{\bm a}]_n$ is a finite-index subgroup of $\Z^d$, or more precisely: 
\begin{align}\label{eq:ReturnTime}
\forall \overleftarrow{g} \in [{\bm a}]_n, \quad \{{\bm n} \in \Z^d: {\bm n}{\bm +} \overleftarrow{g} \in [{\bm a}]_n\} = Z_n.
\end{align}
This observation is the base to show that an odometer system  is a minimal equicontinuous system. This also shows that the action is aperiodic since the intersection of all the $Z_n$ is trivial. 
	
We will be particularly concerned by a special case of odometers: namely the \emph{constant-base} ones. In these systems, $Z_{n}=L^{n}(\Z^{d})$ for each $n \ge 0$, where $L\in \mathcal{M}(d,\Z)$ is an expansion matrix. Recall that an integer matrix $L\in \mathcal{M}(d,\Z)$ is an \emph{expansion} if the norm of each eigenvalue is greater than 1. To simplify notation, we denote the constant-base odometer $\overleftarrow{\Z^{d}}_{(L^{n} (\Z^d))}$ as $\overleftarrow{\Z^{d}}_{(L^{n})}$.

The next result characterizes the factor odometer systems of a fixed odometer system.
\begin{lemma}\cite[Lemma 1]{cortez2006toeplitz}\label{CharacterizationFactorOdometer}
		Let $\overleftarrow{\Z^{d}}_{(Z_{n}^{j})}$ be two odometer systems $(j=1,2)$. There exists a factor map $\pi:\overleftarrow{\Z^{d}}_{(Z_{n}^{1})}\to \overleftarrow{\Z^{d}}_{(Z_{n}^{2})}$ if and only if and only if for every $n \in \NN$ there exists some $m \in \NN$ such that $Z_{m}^{1}\leqslant Z_{n}^{2}$.
	\end{lemma} 
	
\subsubsection{Normalizer condition}
	The proof of \cref{CharacterizationFactorOdometer} can be modified to provide a characterization for the matrices $M\in GL(d,\Z)$ defining a $\GL$-\textit{epimorphism} between two odometer systems $\phi:\overleftarrow{\Z^{d}}_{(Z_{n}^{1})}\to \overleftarrow{\Z^{d}}_{(Z_{n}^{2})}$. 
	
	\begin{lemma}\label{LemmaNessesaryConditionNormalizerOdometer}   Set $M\in GL(d,\Z)$. There exists a continuous map 
		$\phi:\overleftarrow{\Z^{d}}_{(Z_{n}^{1})}\to \overleftarrow{\Z^{d}}_{(Z_{n}^{2})}$, such that $$\forall {\bm n} \in \Z^d,\overleftarrow {g} \in \overleftarrow{\Z^{d}}_{(Z_{n}^{1})},
		\quad \phi ({\bm n}{\bm +} (\overleftarrow {g} )) = M{\bm n }{\bm +} \phi(\overleftarrow {g} ),$$
 if and only if 
		\begin{equation}\label{normalizercondition}\tag{Epimorphism Condition}
		\forall n \in \NN, \exists m(n)\in\NN \textrm{ s.t. }    MZ_{m(n)}^{1}\leqslant Z_{n}^{2}.
		\end{equation}
	\end{lemma}
	
	It follows from \cref{LemmaNessesaryConditionNormalizerOdometer} that a matrix $M\in GL(d,\Z)$ belongs to the linear representation semigroup $\vec{N}(\overleftarrow{\Z^{d}}_{(Z_{n})})$ of an odometer system if and only if it satisfies the following condition, which we called \emph{Normalizer condition}
\begin{equation}\label{normalizercondition1}\tag{NC 1}
		\forall n \in \NN, \exists m(n)\in\NN \textrm{ s.t. }    MZ_{m(n)}\leqslant Z_{n}.
		\end{equation}
		
	\begin{proof}
	For the sufficiency, assume that $M\in GL(d,\Z)$ satisfies \eqref{normalizercondition}. Since the sequences $\{Z_{n}^{i}\}_{n>0}$, $i=1,2$ are decreasing, we may assume that $m(n)\leq m(n+1)$ for all $n\in\NN$. Thus, we have homomorphisms $\phi_{m(n)}^{M}:\Z^{d}/Z_{m(n)}^{1}\to \Z^{d}/Z_{n}^{2}$, given by $\phi_{m(n)}({\bm m}\ (\text{mod}\ Z_{m(n)}^1) )=M{\bm m}\ (\text{mod}\ Z^2_n)$. To finish the proof, we only have to remark that ${\phi^{M}:\overleftarrow{\Z^{d}}_{(Z_{n}^{1})}\to\overleftarrow{\Z^{d}}_{(Z_{n}^{2})}}$ defined as ${\phi(({\bm g}_{n})_{n\in\NN})}=(\phi_{m(n)}({\bm g}_{m(n)}))_{n\in \NN}$ is an $M$-epimorphism.

We now prove the necessity. Let $\phi:\overleftarrow{\Z^{d}}_{(Z_{n}^{1})}\to \overleftarrow{\Z^{d}}_{(Z_{n}^{2})}$ be an $M$-epimorphism. By continuity, for any $n\in \NN$ and ${\bm g}\in \Z^{d}/Z_{n}^{2}$, there exists $m\in \NN$ and ${\bm f}\in \Z^{d}/Z_{m}^{1}$ such that ${[{\bm f}]_{m}\subseteq \phi^{-1}([{\bm g}]_{n})}$. Set ${\bm h}\in Z_{m}^{1}$. For all $\overleftarrow{f}\in [{\bm f}]_{m}$, we have by \eqref{eq:ReturnTime} that ${\bm h}{\bm +}\overleftarrow{f}\in [{\bm f}]_{m}$, which implies that $\phi({\bm h}{\bm +} \overleftarrow{f})=M{\bm h} {\bm +}\phi(\overleftarrow{f})\in [{\bm g}]_{n}$. Since $\phi(\overleftarrow{f})$ is in $[{\bm g}]_{n}$, the set $\left\{{\bm m}\in \Z^{d}\colon {\bm m}{\bm +}\phi(\overleftarrow{f})\in [{\bm g}]_{n}\right\}$ is equal to $Z_{n}^{2}$ (by \eqref{eq:ReturnTime}), which implies that $M{\bm h}\in Z_{n}^{2}$.
	\end{proof}

Notice that since the sequences $\{Z_n^i\}$, $i=1,2$ are decreasing, the \eqref{normalizercondition} is satisfied for any large enough $m\in\NN$ provided it is true for one $m$. This remark implies that the set of matrices $M$ (not necessarily invertible) satisfying the condition \eqref{normalizercondition1} is stable under product and sum, so it is a ring.  
By applying this remark we get the following result

\begin{corollary}\label{cor:CorollariesNormalizerConditionOdometer1}

The semigroup ${N}(\overleftarrow{\Z^{d}}_{(Z_{n})})$ of $\GL$-endomorphims of an odometer $\overleftarrow{\Z^{d}}_{(Z_{n})}$ is a group. 

In particular any $\GL$-endomorphim is invertible and the linear representation semigroup  $\vec{N}(\overleftarrow{\Z^{d}}_{(Z_{n})})$ is a group.    
\end{corollary}
\begin{proof}
Recall that odometer systems are equicontinuous and, hence, are coalescent \cite{auslander1988minimal}. Thus, from \cref{CorollariesNormalizerConditionOdometer}, to show that any $\GL$-endomorphism of an odometer system is invertible, it is enough to show that $\vec{N}(\overleftarrow{\Z^{d}}_{(Z_{n})})$ is a group.

Since $\vec{N}(\overleftarrow{\Z^{d}}_{(Z_{n})})$ is a semigroup, we only have to prove that any element $M\in \vec{N}(\overleftarrow{\Z^{d}}_{(L^{n})})$ admits an inverse inside $\vec{N}(\overleftarrow{\Z^{d}}_{(Z_{n})})$. Since the sets of matrices satisfying \eqref{normalizercondition} is a ring, any integer polynomial of $M$ also satisfies \eqref{normalizercondition}.
Now, the Cayley-Hamilton theorem implies that $M^{d}=\sum\limits_{k=0}^{d-1}b_{k}M^{k}$, where $b_{k}\in \Z$ are the coefficients of the characteristic polynomial of $M$. Notice that $b_{0}=(-1)^{d}\det(M)$, where $\det(M)=\pm 1$, because $M\in \GL$ so that $b_{0}\in \{-1,1\}$. Multiplying by $M^{-1}$, we conclude that $M^{-1}$ can be written as an integer polynomial on $M$. Hence, $M^{-1}$ satisfies \eqref{normalizercondition}, and by \cref{LemmaNessesaryConditionNormalizerOdometer}, we conclude that $M^{-1}\in \vec{N}(\overleftarrow{\Z^{d}}_{(Z_{n})})$.
	\end{proof}

Recall that the automorphisms of the odometer system are well known: they are the translations on it \cite{auslander1988minimal}.  
\begin{lemma}\label{lem:DescriptAutEquicont} For any odometer system we have that
  $$\Aut(\overleftarrow{\Z^{d}}_{(Z_{n})})  = \{\overleftarrow {g} \in \overleftarrow{\Z^{d}}_{(Z_{n})} \mapsto \overleftarrow {g} +\overleftarrow {h} \in \overleftarrow{\Z^{d}}_{(Z_{n})}:  {\overleftarrow {h} \in \overleftarrow{\Z^{d}}_{(Z_{n})}}\}.$$
In particular  $\Aut(\overleftarrow{\Z^{d}}_{(Z_{n})})$ is an abelian group isomorphic to $\overleftarrow{\Z^{d}}_{(Z_{n})}$.
\end{lemma}

As a direct consequence of \cref{cor:CorollariesNormalizerConditionOdometer1} we get the following algebraic structure of the normalizer group of odometer systems.
	
	\begin{corollary}\label{cor:CorollariesNormalizerConditionOdometer2} The normalizer group $N(\overleftarrow{\Z^{d}}_{(Z_{n})})$ of an odometer system is isomorphic to a semidirect product between the odometer system $\overleftarrow{\Z^{d}}_{(Z_{n})}$ and the linear representation group $\vec{N}(\overleftarrow{\Z^{d}}_{(Z_{n})})$.
	\end{corollary}
	
\begin{proof}
Recall from \cref{LemmaNessesaryConditionNormalizerOdometer} that for each $M\in \vec{N}(\overleftarrow{\Z^{d}}_{(Z_{n})})$, one can associate a $M$-isomorphism $\phi^M$ of ${\Z^{d}}_{(Z_{n})}$ defined for any $({\bm g}_{m(n)}\ (\text{mod}\ Z_{m(n)}))_{n\in \NN} \in \overleftarrow{\Z^{d}}_{(Z_{n})}$ by 
$${\phi^M (({\bm g}_{m(n)}\ (\text{mod}\ Z_{m(n)}))_{n\in \NN})}=(M{\bm g}_{n}\ (\text{mod}\ Z_{n}))_{n\in \NN}.$$
Notice that the set $\{\phi^{M}\colon M\in \vec{N}(\overleftarrow{\Z^{d}}_{(Z_{n})})\}$ is a group and defines a group homomorphism $h: \vec{N}(\overleftarrow{\Z^{d}}_{(Z_{n})})\to N(\overleftarrow{\Z^{d}}_{(Z_{n})})$ such that $j\circ h$ is the identity in $\{\phi^{M}\colon M\in \vec{N}(\overleftarrow{\Z^{d}}_{(Z_{n})})\}$, so the exact sequence \eqref{ExactSequenceForNormalizer} is split exact.
	\end{proof}
	
	So, to study the normalizer group of an odometer system, we just have to determine its linear representation group.
Actually, all these results lead us to consider the following question:
\begin{question}\label{ques:realizationOdometer}
Give a characterization of the groups of the form  $\vec{N}(\overleftarrow{\Z^{d}}_{(Z_{n})})$ for any odometer 
$\overleftarrow{\Z^{d}}_{(Z_{n})}$.
\end{question}
We do not answer this question, but we provide necessary conditions for specific odometers: the universal and the constant-base ones in Sections \ref{sec:UnivOdometer} and \ref{sec:ConstantBaseOdometer} respectively.

\subsection{Symbolic dynamics} We recall here classical definitions and we fix the notations on multidimensional subshifts. 

    \subsubsection{Basic definitions} Let $\A$ be a finite alphabet and $d\geq 1$ be an integer. We define a topology on $\A^{\Z^{d}}$ by endowing $\A$ with the discrete topology and considering in $\A^{\Z^{d}}$ the product topology, which is generated by cylinders. Since $\A$ is finite, $\A^{\Z^{d}}$ is a metrizable compact space. In this space, the group $\Z^{d}$ acts by translations (or shifts), defined for every ${\bm n}\in \Z^{d}$ as
$$S^{{\bm n}}(x)_{{\bm k}}=x_{{\bm n}+{\bm k}},\ x= (x_{\bm k})_{\bm k}\in \A^{\Z^{d}},\ {\bm n, \bm k}\in \Z^{d}.$$
	
	
Let $P\subseteq \Z^{d}$ be a finite subset. A \emph{pattern} is an element $\texttt{p}\in \A^{P}$. We say that $P$ is the \emph{support} of $\texttt{p}$, and we denote $P=\supp(\texttt{p})$. A pattern \emph{occurs in} $x\in \A^{\Z^{d}}$, if there exists ${\bm n}\in \Z^{d}$ such that $\texttt{p}=x|_{{\bm n}+P}$ (identifying ${\bm n}+P$ with $P$ by translation). In this case, we denote it $\texttt{p}\sqsubseteq x$ and we call this ${\bm n}$ an \emph{occurrence in} $x$ of $\texttt{p}$.
	
A \emph{subshift} $(X,S,\Z^{d})$ is given by a closed subset $X\subseteq \A^{\Z^{d}}$ which is invariant by the $\Z^{d}$-action. A subshift defines a  \emph{language}. For a finite subset $P\Subset \Z^{d}$ we define
$$\mathcal{L}_{P}(X)=\{\texttt{p}\in \A^{P}: \exists x \in X,\ \texttt{p}\sqsubseteq x\}.$$
	
The \emph{language} of a subshift $X$ is defined as
$$\mathcal{L}(X)=\bigcup\limits_{P\Subset \Z^{d}}\mathcal{L}_{P}(X).$$
	
Let $(X,S,\Z^{d})$ be a subshift and $x\in X$. We say that ${\bm p}\in \Z^{d}$ is a \emph{period} of $x$ if for all ${\bm n}\in \Z^{d}$, $x_{{\bm n}+ {\bm p}}=x_{{\bm n}}$. The subshift $(X,S,\Z^{d})$ is said to be \emph{aperiodic} if there are no nontrivial periods.

Let $\B$ be another finite alphabet and $Y\subseteq \B^{\Z^{d}}$ be a subshift. For $P\Subset \Z^{d}$, we define a $P$-\emph{block map} as a map of the form $\Phi: \mathcal{L}_{P}(X)\to \B$. This induces a factor map $\phi:X\to Y$ given by
$$\phi(x)_{{\bm n}}= \Phi(x|_{{\bm n}+ P}).$$
	
The map $\phi$ is called the \emph{sliding block code} induced by $\Phi$ and $P$ is the support of the map $\phi$. In most of the cases we may assume that the support of the sliding block codes is a ball of the form $B({\bm 0},r)$, where $r$ is a positive integer. We define the \emph{radius}, denoted as $r(\phi)$, as the smallest positive integer $r$ for which a $B({\bm 0},r)$-block map can be defined to induce $\phi$. The next theorem characterizes the factor maps between two subshifts.
	
\begin{theorem}[Curtis-Hedlund-Lyndon] Let $(X,S,\Z^{d})$ and $(Y,S,\Z^{d})$ be two subshifts. A map $\phi:(X,S,\Z^{d})\to (Y,S,\Z^{d})$ is a factor map if and only if there exists a $B({\bm 0},r)$-block map $\Phi:\mathcal{L}_{B({\bm 0},r)}(X)\to \mathcal{L}_{1}(Y)$, such that $\phi(x)_{{\bm n}}=\Phi(x|_{{\bm n}+B({\bm 0},r)})$, for all ${\bm n}\in \Z^{d}$ and $x\in X$.
\end{theorem}
	
For $\GL$-epimorphisms there is a similar characterization. See \cite[Theorem 2.7]{cabezas2023homomorphisms} for a proof.
	
\begin{theorem}[Curtis-Hedlund-Lyndon theorem for $\GL$-epimorphisms]\label{thm:CHLEpimorphism} Let $(X,S,\Z^{d})$ and $(Y,S,\Z^{d})$ be two subshifts and ${M\in \GL}$. A map $\phi:(X,S,\Z^{d})\to (X,S,\Z^{d})$ is an $M$-endomorphism if and only if there exists a $B({\bm 0},r)$-block map $\Phi:\mathcal{L}_{B({\bm 0},r)}(X)\to \mathcal{L}_{1}(Y)$, such that $\phi(x)_{{\bm n}}=\Phi(x|_{M^{-1}{\bm n}+B({\bm 0},r)})$, for all ${\bm n}\in \Z^{d}$ and $x\in X$.
\end{theorem}

This means, for any $\GL$-epimorphism $\phi$, we can define a \emph{radius} (also denoted by $r(\phi)$), as the infimum of $r\in\NN$ such that we can define a $B({\bm 0},r)$-block map which induces it. In the case $r(\phi)=0$, we say that $\phi$ is induced by a \emph{letter-to-letter map}.

\subsubsection{Substitutive subshifts}\label{sec:SubstSubshift} 
We provide a brief overview of multidimensional substitutive subshifts of constant-shape that will be used throughout this article. We refer to \cite{cabezas2023homomorphisms} for basic properties on this topic, where we follow the same notations. Let $L\in \mathcal{M}_{d}(\Z)$ be an integer expansion matrix, $F\subseteq \Z^{d}$ be a fundamental domain of $L(\Z^{d})$ in $\Z^{d}$, i.e., a set of representative classes of $\Z^{d}/L(\Z^{d})$ (with ${\bm 0}\in F$) and $\A$ be a finite alphabet. A \emph{constant-shape substitution} is a map $\zeta:\A\to\A^{F}$. We say that $F$ is the \emph{support} of the substitution. 
Since any element $\bm n \in \Z^d$ can be expressed uniquely as $\bm n = L(\bm j) + \bm f$, with $\bm j \in \Z^d$ and $\bm f \in F$, the substitution extends to $\A^{\Z^d}$ as 
$$ \zeta(x)_{L(\bm j) + \bm f} = \zeta(x_{\bm j})_{\bm f}.$$  

For any $n>0$, we define the $n$-th iteration of the substitution $\zeta^{n}:\A\to \A^{F_{n}}$ by induction $\zeta^{n+1}=\zeta\circ \zeta^{n}$, where the supports of these substitutions satisfy the recurrence $F_{n+1}=L(F_{n})+F_{1}$ for all $n\geq 1$. We will always assume that the sequence of supports $(F_{n})_{n>0}$ is \emph{F\o lner}, i.e., for all ${\bm n}\in \Z^{d}$ we have that
$$\lim\limits_{n\to \infty} \dfrac{|F_{n} \triangle (F_{n} +{\bm n})|}{|F_{n} |}=0.$$

The supports do not need to cover all the space. Nevertheless, up to adding a finite set and taking its images under power of the expansion matrix $L$, they cover the space. This property is explained in the following proposition. It is similar to the notion of remainder in numeration theory and will be technically useful.
\begin{proposition}\cite[Proposition 2.10]{cabezas2023homomorphisms}\label{FiniteSubsetFillsZd}
    	Let $\zeta$ be a constant-shape substitution. Then, the set $K_{\zeta}=\bigcup\limits_{m>0}((\Id_{\R^{d}} -L^{m})^{-1}(F_{m})\cap \Z^{d})$ is finite and satisfies		$$\bigcup\limits_{n\geq 0} L^{n}(K_{\zeta})+F_{n}=\Z^{d},$$
    	
    	\noindent using the notation $F_{0}=\{0\}$.
    \end{proposition}

The \emph{language} of a substitution is the set of all patterns that occur in $\zeta^{n}(a)$, for some $n>0$, $a\in \A$, i.e.,
$$\mathcal{L}_{\zeta}=\{\texttt{p}\colon \texttt{p}\sqsubseteq \zeta^{n}(a),\ \text{for some }n>0,\ a\in \A\}.$$
A substitution $\zeta$ is called \emph{primitive} if there exists a positive integer $n>0$, such that for every $a,b\in \A$, $b$ occurs in $\zeta^{n}(a)$.
If $\zeta$ is a primitive constant-shape substitution, the existence of \emph{periodic points} is well-known, i.e., there exists at least one point $x_{0}\in X_{\zeta}$ such that $\zeta^{p}(x_{0})=x_{0}$ for some $p>0$. In the primitive case, the subshift is preserved by replacing the substitution with a power of it; that is, $X_{\zeta^{n}}$ is equal to $X_{\zeta}$ for all $n>0$. Consequently, we may assume the existence of at least one fixed point. In other words, there exists a point $x \in X_{\zeta}$ such that $x = \zeta(x)$. As in the one-dimensional case, it is important to note that the number of periodic points (or, if we consider proper iterations, the number of fixed points) is finite.

The substitutive subshift $(X_{\zeta},S,\Z^{d})$ is the topological dynamical system, where $X_{\zeta}$ is the set of all sequences $x\in \A^{\Z^{d}}$ such that every pattern occurring in $x$ is in $\mathcal{L}_{\zeta}$. When the substitutive subshift $(X_{\zeta},S,\Z^{d})$ is aperiodic, the substitution satisfies a combinatorial property called \emph{recognizability} \cite{cabezas2023homomorphisms,solomyakrecognizability}.

\begin{definition}
	Let $\zeta$ be a primitive substitution and $x\in X_{\zeta}$ be a fixed point. We say that $\zeta$ is \emph{recognizable on $x$} if there exists some constant $R>0$ such that for all ${\bm i}, {\bm j}\in \Z^{d}$,
	$$x|_{B(L_{\zeta}({\bm i}),R)\cap \Z^{d}}=x|_{B({\bm j},R)\cap \Z^{d}} \implies (\exists {\bm k}\in \Z^{d}) (({\bm j}=L_{\zeta}({\bm k}))\wedge (x_{{\bm i}}=x_{{\bm k}})).$$
\end{definition}

The recognizability property implies some topological and combinatorial properties of the substitutive subshift that we summarize in the following:
\begin{itemize}
    \item The substitutive subshift $(X_{\zeta},S,\Z^{d})$ is aperiodic,
    \item for any $n>0$, the map $\zeta^{n}:X_{\zeta}\to \zeta^{n}(X_{\zeta})$ is a homeomorphism,
    \item for any $n>0$, every $x\in X_{\zeta}$ can be written in a unique way $x=S^{{\bm f}}\zeta^{n}(x_{1})$ with ${\bm f}\in F_{n}$ and $x_{1}\in X_{\zeta}$.
\end{itemize}   
It follows that the map $\pi_n \colon X_\zeta \to F_n $ by $\pi_n(x) = {\bm f} $ when $x \in S^{\bm f} \zeta^n(X_\zeta)$ is well-defined, continuous and can be extended to a factor map $\pi \colon (X_\zeta,S,\Z^{d}) \to (\overleftarrow{\Z^{d}}_{(L^{n})},{\bm +},\Z^{d})$ defined as $\pi (x) = (\pi_n(x))_n$ \cite{cabezas2023homomorphisms}.

\section{Description of the linear representation group of odometer systems}\label{sec:HomoZ2Odometers}

In this section, we describe the linear representation group and its elements for several odometers, specifically the universal and constant-base $\Z^2$-odometer systems. It is worth noting that we did not find a similar result in the existing literature.
One of our main tools involves the characterization \eqref{normalizercondition1} for matrices in  the  linear representation group of an odometer system $\overleftarrow{\Z^{d}}_{(Z_{n})}$.  

We begin by presenting an equivalent formulation of \eqref{normalizercondition1} in terms of arithmetical equations. For any given $n\in \NN$, let $L_{n}\in \mathcal{M}(d,\Z)$ be a matrix such that $L_{n}(\Z^{d})=Z_{n}$. It is important to note that this matrix is unique, up to a (right) composition with a matrix in $\GL$. Then, the condition  \eqref{normalizercondition1} is equivalent to: for all $n\in \NN$, there exists $m_{M}(n)\in \NN$ such that $L_{n}^{-1}ML_{m_{M}(n)}$ is an endomorphism in $\Z^{d}$. Since $\det(L)L^{-1}=\adj(L)$, where $\adj(L)$ is the \emph{adjugate matrix} of $L$, we can express \eqref{normalizercondition1} equivalently as:
	\begin{equation}\label{normalizercondition2}\tag{NC 2}
		\forall n\in \NN, \exists m_{M}(n)\in\NN,\ \adj(L_{n})ML_{m_{M}(n)} \equiv 0\ (\text{mod}\ \det(L_{n})).
	\end{equation}

\subsection{The universal $\Z^{d}$-odometer case}\label{sec:UnivOdometer} 
Let $(\Gamma_{n})_{n\in \NN}$ be an enumeration of all finite-index subgroups of $\Z^{d}$. We define the \emph{universal} $d$-\emph{dimensional odometer system} as follows: Start with $\Lambda_{0}=\Gamma_{0}$, and for any $n\geq1$ set $\Lambda_{n}=\Lambda_{n-1}\cap \Gamma_{n}$. Since the intersection of finite-index subgroups remains a finite-index subgroup, we can define the universal $d$-dimensional odometer as $\overleftarrow{\Z^{d}}_{(\Lambda_{n})}$. This odometer is universal in the sense that, by \cref{CharacterizationFactorOdometer}, any odometer system is a topological factor of the universal odometer. For example, the universal 1-dimensional odometer is equal to $\overleftarrow{\Z}_{(n!\Z)}$. With respect to its linear representation group, \eqref{normalizercondition2} leads to the following result.
	
\begin{proposition}\label{SymmetrySemigroupUniversalOdometer}
The linear representation group of the $d$-dimensional universal odometer is equal to $\GL$.
\end{proposition}
	
\begin{proof}
Consider $L_{n}\in \mathcal{M}(d,\Z)$ such that $L_{n}(\Z^{d})=\Lambda_{n}$. A matrix $M\in \GL$ is in $\vec{N}(\overleftarrow{\Z^{d}}_{(\Lambda_{n})})$ if and only if $M$ satisfies \eqref{normalizercondition2}. Now, for any $n\in \NN$, we can choose $m(n)\in\NN$ large enough such that $\Lambda_{m(n)}\leqslant \det(L_{n})\Z^{d}$. This implies that $\adj(L_{n})ML_{m(n)}\equiv 0\ (\text{mod}\ \det(L_{n}))$ for any matrix $M\in \GL$. We then conclude that $\vec{N}(\overleftarrow{\Z^{d}}_{(\Lambda_{n})})=\GL$.
	\end{proof}

\subsection{The constant-base $\Z^{2}$-odometer case}\label{sec:ConstantBaseOdometer} 

We will be  mainly interested in $\GL$-endomorphisms of constant-base odometers, \ie, when $Z_n=  L^n (\Z^d)$ for each $n\in \NN$ and for some expansion matrix $L$. In this case we get the following direct corollary of \cref{LemmaNessesaryConditionNormalizerOdometer} and the condition \eqref{normalizercondition2}. 

	\begin{corollary} \label{CorollariesNormalizerConditionOdometer} 
	Let $L \in \mathcal{M}(d,\Z)$ be an expansion matrix. 
		\begin{enumerate}
        \item A matrix $M\in \GL$	is in $\vec{N}(\overleftarrow{\Z^{d}}_{(L^{n})})$, if and only if 
		\begin{equation}\label{normalizercondition3}\tag{NC 3}
		\forall n\in \NN, \exists m(n)\in\NN,\ \adj(L^n)ML^{m(n)} \equiv 0\ ({mod}\ \det(L^n)).
	\end{equation}
		
		    \item If $M\in \GL$ commutes with some power of the expansion matrix $L$, then $M$ is in the linear representation semigroup $\vec{N}(\overleftarrow{\Z^{d}}_{(L^{n})})$.
		    
			\item For any $M\in \GL$, we have that $\vec{N}(\overleftarrow{\Z^{d}}_{(ML^{n}M^{-1})})=M\vec{N}(\overleftarrow{\Z^{d}}_{(L^{n})})M^{-1}$.
		\end{enumerate}
	\end{corollary}
	
In the next theorem, we present the structure of the linear representation group of constant-base $\Z^{2}$-odometer systems based on computable arithmetical conditions of the expansive matrix $L$. 
Within this family, we obtain a  bifurcation phenomenon at the level of the linear representation group, depending on arithmetic relations of the coefficients of the matrix $L$. To describe the different cases, we introduce some additional notations. For every integer with $|n|>1$, the \emph{radical} $\rad(n)$ \emph{of} $n$ is defined as the product of the distinct prime numbers that divide $n$.
To describe algebraically the normalizer, we will say that a group $G$ is \emph{virtually} $\Z$ when it admits an infinite cyclic subgroup of finite index. 
The \emph{centralizer} $\Cent_{\GLtwo}(L)$ \emph{of a matrix} $L$ \emph{in} $\GLtwo$ is defined as the subgroup consisting of all matrices in $\GLtwo$ commuting with $L$. Recall that, as established in \cref{CorollariesNormalizerConditionOdometer}, the centralizer $\Cent_{\GLtwo}(L)$ is always a subgroup of $\vec{N}(\overleftarrow{\Z^{d}}_{(L^{n})})$.

\begin{theorem}\label{GeneralTheoremDifferentCasesNormalizer}
	Let $L \in \mathcal{M}(2,\Z)$ be an integer expansion matrix.
	
	\begin{enumerate}[label=(\arabic*), ref=(\arabic*)]
		\item\label{RadDividesEverythingCase} If $\rad(\det(L))$ divides $\trace(L)$, then the linear representation group $\vec{N}(\overleftarrow{\Z^{2}}_{(L^n)})$ is equal to $\GLtwo$.
		
		\item Otherwise
		
		\begin{enumerate}[label=(\alph*),ref=(\alph*)]
			\item\label{it:case2a} If the spectrum of the matrix $L$ is disjoint from the integers, then the linear representation group $\vec{N}(\overleftarrow{\Z^{2}}_{(L^n)})$ is the centralizer $\Cent_{GL(2,\Z)}(L)$. 
			
			Moreover, if the spectrum of $L$ is disjoint from the real line, then $\vec{N}(\overleftarrow{\Z^{2}}_{(L^n)})$ is a finite group.
			
			\item\label{it:case2b} When the spectrum of $L$ contains an integer value, the linear representation group $\vec{N}(\overleftarrow{\Z^{2}}_{(L^n)})$ is finite or virtually $\Z$. 
			
			More precisely, under explicit arithmetical conditions of $L$, $\vec{N}(\overleftarrow{\Z^{2}}_{(L^n(\Z^{2}))})$  isomorphic to $\Z/ 2\Z$ or $\Z^{2}/(2\Z\times2\Z)$, or its commutator subgroup is cyclic  and  in this last case its abelianization is finite. 
		\end{enumerate}
	\end{enumerate}
\end{theorem}

Along the proof, the group structure of $\vec{N}(\overleftarrow{\Z^{2}}_{(L^n(\Z^{2}))})$ is specified in terms of the arithmetical properties of the coefficient of $L$. 
 
The following examples illustrate the different cases of \cref{GeneralTheoremDifferentCasesNormalizer} according to the expansion matrix $L$.

\begin{example}[Different results for \cref{GeneralTheoremDifferentCasesNormalizer}]\label{ExamplesGeneralTheoremNormalizerGroupOdometer}
	\begin{enumerate}[label=(\arabic*), ref=(\arabic*)]
	    \item As we will see in the proof of \cref{GeneralTheoremDifferentCasesNormalizer}, the case \ref{RadDividesEverythingCase} can be easily generalized for higher dimensions in the following way: If $\rad(\det(L))$ divides every coefficient of the characteristic polynomial of $L$, then the linear representation semigroup $\vec{N}(\overleftarrow{\Z^{d}}_{(L^{n})})$ is equal to $\GL$. In particular, if $L=pM$, with $p\in \Z$ large enough so that the eigenvalues of $L$ have absolute value strictly greater than $1$ and $M\in \GL$, then the linear representation group $\vec{N}(\overleftarrow{\Z^{d}}_{L^{n}})$ is $\GL$.

		\item The matrix $L_{\theLvariable}=\begin{pmatrix}
			2 & -1 \\ 1 & 5
		\end{pmatrix}$ illustrates the case  (2)\ref{it:case2a}: $\trace(L_{\theLvariable})=7$, $\det(L_{\theLvariable})=11$, and $L_{\theLvariable}$ has real eigenvalues (which are equal to $7/2\pm\sqrt{5}/2$). The matrices in $\vec{N}(\overleftarrow{\Z^{2}}_{(L_{\theLvariable}^{n})})$ are the ones commuting with $L_{\theLvariable}$, which is an infinite group containing $\begin{pmatrix}
			2&1\\-1&-1
		\end{pmatrix}$.
		\stepcounter{Lvariable}
		\item\label{ComplexEigenvalues}The matrix   $L_{\theLvariable}=\begin{pmatrix}
			2 & -1 \\ 1 & 3
		\end{pmatrix}$ also illustrates the case  (2)\ref{it:case2a} but with a spectrum disjoint from the real line: $\trace(L_{\theLvariable})=5$ and $\det(L_{\theLvariable})=7$, and $L_{\theLvariable}$ has complex eigenvalues $5/2\pm i\sqrt{3}/2$. The linear representation semigroup $\vec{N}(\overleftarrow{\Z^{2}}_{L_{\theLvariable}^{n}})$ is equal to $\Cent_{GL(2,\Z)}(L_{\theLvariable})$, which corresponds to the set
		$$\left\{\begin{array}{ccc}
			\begin{pmatrix}
				1 & 1 \\ -1 & 0
			\end{pmatrix},&\begin{pmatrix}
				-1 & -1 \\ 1 & 0
			\end{pmatrix},&\begin{pmatrix}
				0 & -1 \\ 1 & 1
			\end{pmatrix}\\	
			\begin{pmatrix}
				0 & -1 \\ 1 & -1
			\end{pmatrix},&\begin{pmatrix}
				1 & 0 \\ 0 & 1
			\end{pmatrix},&\begin{pmatrix}
				-1 & 0 \\ 0 & -1
			\end{pmatrix}\\
		\end{array}\right\}.$$
\stepcounter{Lvariable}
	\item The matrix $L_{\theLvariable}=\begin{pmatrix}
		6& 1\\0& 2
	\end{pmatrix}$ illustrates the  case (2)\ref{it:case2b}. This is an upper triangular matrix  which is not diagonalizable by $\GLtwo$. It is be proved that $\vec{N}(\overleftarrow{\Z}_{(L_{\theLvariable}^{n})})$ is conjugate to $\left\{\begin{pmatrix}m_{11} & m_{12}\\ 0 & m_{22}\end{pmatrix}\colon |m_{11}m_{22}|=1, m_{12}\in \Z\right\}$ via the matrix $\begin{pmatrix}
	    1 & 0\\4 & 1
	\end{pmatrix}$, so it is virtually $\Z$. It can be directly checked that the linear representation group $\vec{N}(\overleftarrow{\Z^{d}}_{(L^{n})})$ associated with a matrix $L$ diagonalizable by $\GLtwo$ also has the same group structure, being isomorphic to a set of invertible upper triangular matrices. 
\stepcounter{Lvariable}
		\item The matrix $L_{\theLvariable}=\begin{pmatrix}
			3& 1\\0& 7
		\end{pmatrix}$ also concerns the case (2)\ref{it:case2b}. This matrix has eigenvalues $3$ and $7$. In the proof of \cref{GeneralTheoremDifferentCasesNormalizer}  it is shown that a matrix $M$ is in $\vec{N}(\overleftarrow{\Z^{2}_{(L^{n}}})$ if and only if $M$ commutes with $L_{\theLvariable}$, so $\vec{N}(\overleftarrow{\Z^{2}}_{(L_{\theLvariable}^{n})})$ is isomorphic to $\Z/2\Z$.
\stepcounter{Lvariable}
		\item Our last example illustrating the case (2)\ref{it:case2b} is the  matrix $L_{\theLvariable}=\begin{pmatrix}
			2& 1\\0& 3
		\end{pmatrix}$ with  eigenvalues  $2$ and $3$. It is also shown that $\vec{N}(\overleftarrow{\Z^{2}}_{(L_{\theLvariable}^{n})})=\Cent_{\GLtwo}(L_{\theLvariable})$, which is isomorphic to $(\Z/2\Z)^2$.
		
	\end{enumerate}
\end{example}

\stepcounter{Lvariable}
\begin{remark}\label{rem:compatibilityCounterExample} Note that
 \cref{GeneralTheoremDifferentCasesNormalizer} implies that the factor map between equicontinuous systems is not necessarily compatible with $\GL$-endomorphisms. Consider $X$ as the universal $\Z^{2}$-odometer, and set $Y=\overleftarrow{\Z^{2}}_{(L_{1}^{n})}$. Hence $(Y,{\bm +},\Z^{d})$ is an equicontinuous factor of $(X,{\bm +},\Z^{d})$. Now, by \cref{SymmetrySemigroupUniversalOdometer}, we can define an isomorphism associated with the matrix $\begin{pmatrix}
		2 & 1\\1 & 1
	\end{pmatrix}$ in $X$, However, in contrast, \cref{GeneralTheoremDifferentCasesNormalizer} and \cref{CompatibilityPropertyFactors} establish that such isomorphism is not possible in $Y$.
\end{remark}

\subsection{Proof of \cref{GeneralTheoremDifferentCasesNormalizer}}\label{sec:ProofTheo3.3}

In this subsection, we will prove \cref{GeneralTheoremDifferentCasesNormalizer}. We decompose this proof according to its spectral properties.  We will get more precise  results as the ones stated in \cref{GeneralTheoremDifferentCasesNormalizer}. We start with the case where the expansion matrix has integer eigenvalues. 

From now on, an integer expansion matrix $L$ will be denoted as $L=\begin{pmatrix}
	p & q\\ r & s
\end{pmatrix}$, its powers as $L^{n}=\begin{pmatrix}
	p(n) & q(n)\\ r(n) & s(n)
\end{pmatrix}$ and a matrix $M$ in $\GLtwo$ as $M=\begin{pmatrix}m_{11} & m_{12}\\ m_{21} & m_{22}
\end{pmatrix}$.
		
	\subsubsection{The triangular case}
		We now consider the case where $L$ is a triangular matrix. We focus only on the upper triangular case, i.e., $q\neq0$ and $r=0$. The lower triangular case can be deduced from this, thanks to \cref{CorollariesNormalizerConditionOdometer} via conjugation with the matrix $\begin{pmatrix}
			0 & 1 \\ 1 & 0
		\end{pmatrix}$. For all $n\in \Z$ we have  $L^{n}=\begin{pmatrix}p^{n} & q(n)\\ 0 & s^{n}\end{pmatrix},$ where $q(n)=q(p^{n}-s^{n})/(p-s)=q\sum\limits_{i=0}^{n-1}p^{i}s^{n-1-i}$. Since $\det(L)=ps$ and $\trace(L)=p+s$, $\rad(\det(L))$ divides $\trace(L)$ if and only if $\rad(p)=\rad(s)$. In this case we get a more precise result about the linear representation group as the one mentioned in \cref{GeneralTheoremDifferentCasesNormalizer}.
		
		\begin{proposition}\label{uppertriangularcasenormalizer}
			Let $L\in \mathcal{M}(2,\Z)$ be an expansion upper triangular matrix such that $\rad(\det(L))$ does not divide $\trace(L)$. Then, we have one of the following: 
			
			\begin{enumerate}
				\item\label{ConflictivecaseUppertriangularcase} If $\rad(p)$ does not divide $s$ and $\rad(s)$ divides $p$, then a matrix $M\in GL(2,\Z)$ is in $\vec{N}(\overleftarrow{\Z^{2}}_{L^{n}})$ if and only if $(p-s)^{2}m_{12}=m_{21}q^{2}+(p-s)(m_{11}-m_{22})q$. Moreover, $\vec{N}(\overleftarrow{\Z^{2}}_{(L^{n})})$ is virtually $\Z$.
				
				\item Assume that $\rad(p)$ divides $s$ and $\rad(s)$ does not divide $p$. Then $\vec{N}(\overleftarrow{\Z^{2}}_{(L^{n})})$ is virtually $\Z$.
					
					\item If $\rad(p)$ does not divide $s$ and $\rad(s)$ does not divide $p$, we have two cases:
					\begin{itemize}
						\item If $2q\in (p-s)\Z$, then $\vec{N}(\overleftarrow{\Z^{2}}_{(L^{n})})$ is isomorphic to $\Z/2\Z\times \Z/2\Z$.
						\item Otherwise, $\vec{N}(\overleftarrow{\Z^{2}}_{(L^{n})})$ is isomorphic to $\Z/2\Z$.
					\end{itemize}
				\end{enumerate}
			\end{proposition}
	
			\begin{proof}
				Let $M$ be in $\vec{N}(\overleftarrow{\Z^{2}}_{(L^{n})})$. Define the matrix $\overline{M}=(p-s)M-(m_{11}-m_{22})L-(p\cdot m_{22}-m_{11}\cdot s)\Id_{\R^{2}}$. Then, $\overline{M}$ satisfies \eqref{normalizercondition3}. Moreover, note that $\overline{M}$ has the form $\overline{M}=\begin{pmatrix}
					0 & \overline{m_{12}}\\ \overline{m_{21}} & 0
				\end{pmatrix}$, where $\overline{m_{12}}=(p-s)m_{12}-(m_{11}-m_{22})q$ and $\overline{m_{21}}=(p-s)m_{21}$, with $\overline{m_{12}},\overline{m_{21}}\in\Z$. Now, \eqref{normalizercondition3} implies that for all $n,\ell>0$, 
				
				\begin{align}
					\begin{pmatrix}
						-\overline{m_{21}}p^{\ell}q(n) & \overline{m_{12}}s^{n+\ell}-\overline{m_{21}} q(\ell)q(n)\\
						\overline{m_{21}}p^{n+\ell} & \overline{m_{21}}pq(\ell)
					\end{pmatrix} \equiv \begin{pmatrix}
						0 & 0\\
						0 & 0
					\end{pmatrix}\ (\text{mod}\ p^{n}s^{n}).\label{eqsuppertriangularcase}
				\end{align}
				
				Suppose that $\rad(s)$ does not divide $p$. Then, there exists a prime number $t$ dividing $s$ such that for all $n>0$ and $\ell>0$, $p^{\ell}$ is an invertible element in $\Z/t^{n}\Z$. Hence, $\overline{m}_{21}\equiv 0\ (\text{mod}\ t^{n})$ for any $n>0$, which implies that  $\overline{m}_{21}=0$, so $m_{21}=0$. Now, by \eqref{eqsuppertriangularcase}, we get that
				\begin{align}
				\forall n\ge 0, \exists \ell\ge 0, \quad 	\overline{m}_{12}s^{\ell} & \equiv 0 \ (\text{mod}\ p^{n}). \label{eq7uppertriangularcase}
				\end{align}
				
				\noindent There are  two cases:
				
				\begin{itemize}
					\item  If $\rad(p)$ does not divide $s$, then \eqref{eq7uppertriangularcase} implies that $\overline{m}_{12}=0$. We conclude that $\overline{M}=\begin{pmatrix}
						0 & 0\\ 0 & 0
					\end{pmatrix}$, i.e., $(p-s)M=(m_{11}-m_{22})L+(p\cdot m_{22}-m_{11}\cdot s)\Id_{\R^{2}}$. Since $m_{21}=0$, then $M$ has the form
					$$M=\begin{pmatrix}
						m_{11} & m_{12}\\ 0 & m_{22}
					\end{pmatrix},$$
					
					\noindent where $m_{12} \in \Z$ satisfies $(p-s)m_{12}=(m_{11}-m_{22})q$. 
					
					\begin{itemize}
						\item Note that $m_{11}=m_{22}$ if and only if $m_{12}=0$. 
						
						\item If $m_{11}\neq m_{22}$, then $m_{11}-m_{22}\in \{-2,2\}$, so $(p-s)m_{12}=\pm2q$. Since $M$ has integer coefficients, this necessarily implies that $2q\in (p-s)\Z$. If this condition is satisfied, then $M$ has the form
						$$M=\begin{pmatrix}
							m_{11} & \frac{(m_{11}-m_{22})q}{p-s}\\ 0 & m_{22}
						\end{pmatrix}.$$
						
						It is straightforward  to check  that $M^{2}$ is the identity matrix. We conclude that $\vec{N}(\overleftarrow{\Z^{2}}_{(L^{n})})$ is isomorphic to $\Z/2\Z \times \Z/2\Z$. If $2q\notin (p-s)\Z$, then $\vec{N}(\overleftarrow{\Z^{2}}_{(L^{n})})$ is isomorphic to $\Z/2\Z$.
					\end{itemize}
					
					\item If $\rad(p)$ divides $s$, then any $\overline{m}_{12}\in \Z$ satisfies \eqref{eq7uppertriangularcase}. Thus, any matrix $M=\begin{pmatrix}
						m_{11} & m_{12}\\ 0 & m_{22}
					\end{pmatrix}$ with $|m_{11}m_{22}|=1$ satisfies \eqref{normalizercondition3}.
				\end{itemize}
				
				Finally, if $\rad(s)$ divides $p$, then for any $n>0$ and any $\ell$ large enough $s^{n}$ divides $p^{\ell}$ and $q(\ell)$. Let $t$ be a prime number dividing $p$ that does not divide $s$. Then, by \eqref{eqsuppertriangularcase} we obtain that
				\begin{align}
					(p-s)^{2}\overline{m}_{12}s^{n+\ell}\equiv \overline{m}_{21}q^{2}s^{n+\ell} \ (\text{mod}\ t^{n}). \label{eq8uppertriangularcase}
				\end{align}
				
				\noindent Since $t$ does not divide $s$, for any $n,\ell>0$, $s^{n+\ell}$ is an invertible element in $\Z/t^{n}\Z$. So \eqref{eq8uppertriangularcase} is reduced to
				\begin{align}
				\forall n,\	(p-s)^{2}\overline{m}_{12} & \equiv \overline{m}_{22}q^{2} \ (\text{mod}\ t^{n}). \label{eq9uppertriangularcase}
				\end{align}
				
				This implies that $(p-s)^{2}\overline{m}_{12}=\overline{m}_{21}q^{2}$. Thus, we get that \begin{align}
					(p-s)^{2}m_{12}=m_{21}q^{2}+(p-s)(m_{11}-m_{22})q.\label{eq10uppertriangularcase}
				\end{align}
				
				This implies that if $M\in \vec{N}(\overleftarrow{\Z^{2}}_{(L^{n})})$, then $M$ is in $\sspan_{\Q}\left\{ L,\Id_{\R^{d}}, \begin{pmatrix}
				    0 & 1\\ (p-s)^2/q^2 & 0
				\end{pmatrix}\right\}$. We separate here in two cases:
				
				\begin{itemize}
				    \item If $(p-s)$ divides $q$, we write $q=k(p-s)$ for some $k\in \Z$. By \eqref{eq10uppertriangularcase}, we have that
				$$m_{12}=m_{21}k^{2}+k(m_{11}-m_{22}).$$
				
				\noindent Since $|\det(M)|=1$ and $\det(M)=(m_{11}+m_{21}\cdot k)(m_{22}-m_{21}\cdot k)$, we get that $|m_{11}+m_{21}\cdot k|=1$ and $|m_{22}-m_{21}\cdot k|=1$. We can parameterize the matrices in $\vec{N}(\overleftarrow{\Z^{2}}_{(L^{n})})$ as follows:
				$$\left\{
\Id_{\R^2}+\ell X_k,  Y_k + \ell X_k, -Y_k + \ell X_k, -\Id_{\R^2}+\ell X_k
:  \ell \in\mathbb{Z}\right\},$$
where $\displaystyle{X_k = 	\begin{pmatrix}  -k & - k^2 \\ 1 & k  \end{pmatrix} }$ and   $\displaystyle{Y_k = 	\begin{pmatrix} 1 & 2k \\ 0 & -1  \end{pmatrix} }$.
			%
			%
				Note that this group is virtually $\Z$, since the quotient by  the group generated by $\Id_{\R^2} +X_k$ 
				is finite. 
				
				\item If $(p-s)$ does not divide $q$, we will find a matrix $P\in \GLtwo$ such that $\vec{N}(\overleftarrow{\Z^{2}}_{(L^{n})})$ is conjugate to the group of matrices $\left\{\begin{pmatrix} m_{11} & m_{12}\\
				0 & m_{22}\end{pmatrix}\colon m_{11},m_{12},m_{22}\in \Z, |m_{11}m_{22}|=1\right\}$ and we conclude that $\vec{N}(\overleftarrow{\Z^{2}}_{(L^{n})})$ is virtually $\Z$. Indeed, set $c=\gcd(p-s,q)$ and $g=(p-s)/c$, $h=q/c$. Since $\gcd(g,h)=1$, B\'ezout's lemma implies the existence of two numbers $e,f\in \Z$ such that $eh-gf=1$. A standard computation shows that $P=\begin{pmatrix} e & f\\ g & h\end{pmatrix}$ is such a matrix. 
				\end{itemize}

			\end{proof}
				
\subsubsection{The general case}\label{SectionProofOfTheoremGeneralBifurcationNormalizer} We are ready to prove \cref{GeneralTheoremDifferentCasesNormalizer}. 
				
\begin{proof}[Proof of \cref{GeneralTheoremDifferentCasesNormalizer}]
We use the notation introduced in \cref{sec:ProofTheo3.3}. Since we already proved the triangular case, we assume that the coefficients  of the  expansion matrix $L$ satisfy $q\cdot r\neq 0$.

It will be useful to note that the Cayley-Hamilton theorem implies that 
	\begin{align}\label{eq:CayleyHamilton}
	L^{2}=\trace(L)L-\det(L)\Id_{\R^2}.
	\end{align} 
	First assume that $\rad(\det(L))$ divides $\trace(L)$. By \eqref{eq:CayleyHamilton}, we can conclude that ${L^{2}\equiv 0\ (\text{mod}\ \rad(\det(L)))}$. Hence, for all $n\in \NN$ there exists $m(n)\in\NN$ large enough such that $L^{m(n)}\equiv 0\ (\text{mod}\ \det(L)^{n})$. Therefore, any matrix in $\GLtwo$ satisfies \eqref{normalizercondition3} and we can deduce that $\vec{N}(\overleftarrow{\Z^{2}}_{(L^{n})})=GL(2,\Z)$. 
					
	Now we deal with the case when $\rad(\det(L))$ does not divide $\trace(L)$. In dimension $2$, this implies that $L$ is diagonalizable.  Let $M=\begin{pmatrix}m_{11} & m_{12}\\ m_{21} & m_{22}\end{pmatrix}$ be in $\vec{N}(\overleftarrow{\Z^{2}}_{(L^{n})})$, so that it satisfies \eqref{normalizercondition3} . Define the matrix $\overline{M}=rM-m_{21}L-(r\cdot m_{11}-p\cdot m_{21})\Id_{\R^{2}}$. The matrix $\overline{M}$ also satisfies \eqref{normalizercondition3} and has the form $\overline{M}=\begin{pmatrix}
						0 & \overline{m_{12}}\\ 0 & \overline{m_{22}}
					\end{pmatrix}$, with  $\overline{m}_{12},\overline{m}_{21}\in\Z$.
					
Suppose first that $L$ has integer eigenvalues $t_1, t_2\in \Z$, i.e., we can write 
$$(eh-fg)L= P\begin{pmatrix}
			        t_{1} & 0\\
			        0 & t_{2}
			    \end{pmatrix}\adj(P), \quad\text{ for some integer matrix } P= \begin{pmatrix}
			        e & f\\
			        g & h
			    \end{pmatrix}.$$
		\noindent If ${|eh-fg|=1}$, then we can use \cref{uppertriangularcasenormalizer} with \cref{CorollariesNormalizerConditionOdometer} to conclude that $\vec{N}(\overleftarrow{\Z^{2}}_{(L^{n})})$ is conjugate (via $P$ in $\GLtwo$) to the linear representation group $\vec{N}(\overleftarrow{\Z^{2}}_{(t_{1}^{n}\Z\times t_{2}^{n}\Z)})$. The same conclusion holds when $L$ is conjugate (via a $\GLtwo$-matrix) to a triangular matrix. We then assume that $|eh-fg|>1$, $e,f,g,h\in \Z$, and $\gcd(e,g)=\gcd(f,h)=1$. For any $n>0$, the coefficients of $L^n$ are given by:
			    $$\begin{array}{ll}
			     p(n)=\dfrac{eht_{1}^{n}-fgt_{2}^{n}}{eh-fg} & q(n)=\dfrac{ef(t_{1}^{n}-t_{2}^{n})}{eh-fg}\\
			     &\\
			     r(n)=\dfrac{gh(t_{1}^{n}-t_{2}^{n})}{eh-fg} & s(n)=\dfrac{eht_{2}^{n}-fgt_{1}^{n}}{eh-fg}.
			    \end{array}$$
			    
\noindent So, \eqref{normalizercondition3} can be rewritten as:
\begin{align}\label{eqsintegereigenvaluecases} gh(t_{1}^{\ell}-t_{2}^{\ell})[\overline{m}_{12}(eht_{2}^{n}-fgt_{1}^{n})-\overline{m}_{22}ef(t_{2}^{n}-t_{1}^{n})] \equiv 0\ (\text{mod}\ t_{1}^{n}t_{2}^{n})\\   gh(t_{1}^{\ell}-t_{2}^{\ell})[\overline{m}_{22}(eht_{1}^{n}-fgt_{2}^{n})-\overline{m}_{12}gh(t_{1}^{n}-t_{2}^{n})] \equiv 0\ (\text{mod}\ t_{1}^{n}t_{2}^{n})\notag\\       (eht_{2}^{\ell}-fgt_{1}^{\ell})[\overline{m}_{12}(eht_{2}^{n}-fgt_{1}^{n})-\overline{m}_{22}ef(t_{2}^{n}-t_{1}^{n})] \equiv 0\ (\text{mod}\ t_{1}^{n}t_{2}^{n}) \notag\\    (eht_{2}^{\ell}-fgt_{1}^{\ell})[\overline{m}_{22}(eht_{1}^{n}-fgt_{2}^{n})-\overline{m}_{12}gh(t_{1}^{n}-t_{2}^{n})] \equiv 0\ (\text{mod}\ t_{1}^{n}t_{2}^{n}).\notag
			    \end{align}
Since $\rad(\det(L))=\rad(t_{1}t_{2})$ does not divide $\trace(L)=t_{1}+t_{2}$, one of the following three cases hold:
				
\begin{enumerate}[label=\textbf{Case }\arabic*.]
    \item Suppose that $\rad(t_{1})$ divides $t_{2}$, but there exists a prime number $t$ dividing $t_{2}$ that does not divide $t_{1}$. Then \eqref{eqsintegereigenvaluecases} can be reduced to
	      \begin{align}\label{eqsintegereigenvaluecasesCASE2}    fgt_{1}^{m+n}[\overline{m}_{22}\cdot eh-\overline{m}_{12}\cdot gh] \equiv 0\ (\text{mod}\ t^{n}).
			          \end{align}
		Since $t$ does not divide $t_{1}$, for any $n,m\geq 0$,  $t_{1}^{n+m}$ is an invertible element in $\Z/t^{n}\Z$. We can also choose $n$ large enough so that $t^{n}$ does not divide any of the coefficients $e,f,g,h\in \Z$. We then conclude that
$\overline{m}_{12}g=\overline{m}_{22}e$. This implies that 
$$\vec{N}(\overleftarrow{\Z^{2}}_{(L^{n})})\subseteq \left\{aL+b\Id_{\R^{2}}+c\begin{pmatrix}
				    0 & e\\
				    0 & g
				\end{pmatrix}, a,b,c \in \frac 1r\Z \right\}\cap \GLtwo.$$

Since $P^{-1}\begin{pmatrix}
 0 & e     \\ 0 & g
\end{pmatrix} P =  \begin{pmatrix}
 g & h     \\ 0 & 0
\end{pmatrix}$, the set  $P^{-1} \vec{N}(\overleftarrow{\Z^{2}}_{(L^{n})}) P $ is a subgroup of unimodular upper triangular  matrices  in $G$, where
$$ G = \left\{\begin{pmatrix}
				    a & b\\
				    0 & c
				\end{pmatrix}, \quad  a,b,c \in \frac 1r\Z, |ac| =1\right\}.$$ 
Notice that all commutators of $G$ are of the form $\begin{pmatrix}
 1 & b     \\ 0 & 1
\end{pmatrix}$. So, the derived subgroup $G'$ (generated by the commutators) is isomorphic to $\Z$. Moreover, the abelianization  $G/G'$ of $G$ is finite. Therefore, the abelianization of $\vec{N}(\overleftarrow{\Z^{2}}_{(L^{n})})$ is finite, and its derived subgroup is isomorphic to a subgroup (eventually trivial)  of $\Z$. 
Conversely, a direct computation shows that the matrix $P^{-1}\begin{pmatrix}1& \det(P) \\ 0 &1 \end{pmatrix}P$ satisfies \eqref{normalizercondition3}, proving that the derived subgroup of  $\vec{N}(\overleftarrow{\Z^{2}}_{(L^{n})})$ is nontrivial.  

    \item The case where $\rad(t_{2})$ divides $t_{1}$ but $\rad(t_1)$ does not divide $t_2$ is symmetric to the former one. 
				
    \item Neither $\rad(t_{1})$ divides $t_{2}$ nor $\rad(t_{2})$ divides $t_{1}$. The  former computations in the two cases provide that\begin{align*}
        \overline{m}_{12}g = \overline{m}_{22}e,  \quad \wedge \quad 
        \overline{m}_{12}h = \overline{m}_{22}f.
    \end{align*} 
Since $eh-fg\neq 0$, this implies that $\overline{m}_{12}=0$ and $\overline{m}_{22}=0$, so $\overline{M}=0$. We conclude that $M$ commutes with $L$, i.e., the linear representation group $\vec{N}(\overleftarrow{\Z^{2}}_{(L^{n})})$ is equal to $\Cent_{\GLtwo}(L)$. 
A matrix $M$ in $\GLtwo$ commuting with $L$ has to preserve each of the one-dimensional eigenspaces of $L$. Therefore the matrix $M$ is $\Q$-diagonalizable with the same eigenvectors. Since it is in $\GLtwo$, its spectrum is a subset of $\{1, -1\}$. Moreover the spectrum determines only one matrix  in $\Cent_{\GLtwo}(L)$. Since $-M$ is also in the centralizer,  it follows that  the centralizer of $L$ is isomorphic to $\Z/2\Z$ or $(\Z/2\Z)^{2}$.
\end{enumerate}
				
Now we suppose that $L$ does not have integer eigenvalues. A direct induction on \eqref{eq:CayleyHamilton} gives, for any $n>0$, that $L^{n}\equiv \trace(L)^{n-1}L\ (\text{mod}\ \det(L))$. Since $\rad(\det(L))$ does not divide $\trace(L)$, there exists a prime number $t$ dividing $\det(L)$ that does not divide $p$ or $s$. Without loss of generality (up to a conjugation with $\begin{pmatrix}
						0 & 1 \\ 1 & 0
					\end{pmatrix}$) we may assume that $t$ does not divide $s$. As $s(n)\equiv \trace(L)^{n-1}s\ (\text{mod}\ \det(L))$, then, for all $n>0$ and $m>0$, $s(m)$ is an invertible element in $\Z/t^{n}\Z$. Hence, \eqref{normalizercondition3} implies that 
				
				\begin{align}
					\overline{m_{12}}s(n)-\overline{m_{22}}q(n) & \equiv 0 \ (\text{mod}\ t^{n})\label{eq5generalcase}\\
					-\overline{m_{12}}r(n)+\overline{m_{22}}p(n) & \equiv 0 \ (\text{mod}\ t^{n}),\label{eq6generalcase} 
				\end{align}
				
				\noindent which is equivalent to 
				\begin{align}
					\adj(L^{n})\dbinom{\overline{m_{12}}}{\overline{m_{22}}} \equiv \dbinom{0}{0}\ (\text{mod}\ t^{n}).\label{eq7generalcase}
				\end{align} 
				
				Consider the set $E=\left\{\dbinom{\overline{m_{12}}}{\overline{m_{22}}}\in \Z^{2}\colon\ \text{satisfying}\ \eqref{eq7generalcase}\ \text{for all}\ n>0\right\}$. This set is $\adj(L)$-invariant, and if $\overline{m_{22}}=0$, then \eqref{eq5generalcase} implies that $\overline{m_{12}}=0$. 
				
				Now, set $\dbinom{\overline{m_{12}}^{(1)}}{\overline{m_{22}}^{(1)}},\dbinom{\overline{m_{12}}^{(2)}}{\overline{m_{22}}^{(2)}}\in E$. Note that
				$$\overline{m_{22}}^{(2)}\dbinom{\overline{m_{12}}^{(1)}}{\overline{m_{22}}^{(1)}}-\overline{m_{22}}^{(1)}\dbinom{\overline{m_{12}}^{(2)}}{\overline{m_{22}}^{(2)}}=\dbinom{\overline{m_{22}}^{(2)}\overline{m_{12}}^{(1)}-\overline{m_{22}}^{(1)}\overline{m_{12}}^{(1)}}{0},$$
				
				\noindent hence, by the former remark  $\overline{m_{22}}^{(2)}\dbinom{\overline{m_{12}}^{(1)}}{\overline{m_{22}}^{(1)}}=\overline{m_{22}}^{(1)}\dbinom{\overline{m_{12}}^{(2)}}{\overline{m_{22}}^{(2)}}$. So, $E$ is an $\adj(L)$-invariant $\Z$-module of rank at most 1. 
				
\noindent Since $L$ does not have integer eigenvalues, $E$ must have rank 0. This implies that $\overline{m_{12}}=\overline{m_{22}}=0$. Hence $M$ is a polynomial in $L$ and  commutes with $L$. We conclude that $\vec{N}(\overleftarrow{\Z^{2}}_{(L^{n})})$ is equal to $\Cent_{\GLtwo}(L)$.
				
We claim that the centralizer $\Cent_{\GLtwo}(L)$ is finite when $L$ has no real eigenvalues. Indeed, for $M$ to commute with $L$ it has to satisfy	\begin{equation}\label{eq8generalcase}
					\begin{array}{cc}
						r\cdot m_{12}-q\cdot m_{21} & =0\\
						r\cdot m_{22}-s\cdot m_{21}-r\cdot m_{11}+p\cdot m_{21} & = 0.
					\end{array}
				\end{equation}
				
				\begin{itemize}
					\item Suppose $p=s$. In this case, \eqref{eq8generalcase} implies that $m_{11}=m_{22}$ and $m_{21}=m_{11}\cdot r/q$. Note that $L$ has complex eigenvalues if and only if $qr<0$, as determined by the condition $(2p)^{2}-4(p^{2}-qr)<0$. Since $|\det(M)|=1$, then $|m_{11}^{2}-m_{12}^{2}\cdot r/q|$ equals $1$. Therefore, when $qr<0$, there exists a finite number of points $(m_{11},m_{12})\in \Z^{2}$ satisfying \eqref{eq8generalcase}. 
				
				\item If $p\neq s$, then \eqref{eq8generalcase} implies that $m_{12}=q(m_{11}-m_{22})/(p-s)$ and $m_{21}=r(m_{11}-m_{22})/(p-s)$. Since $M \in GL(2,\Z)$, we get that
				\begin{equation}\label{eqgeneralcasefordiagonalelements}m_{11}m_{22}-(m_{11}-m_{22})^{2}\dfrac{qr}{(p-s)^{2}}=\pm1.
				\end{equation}
				
				In this case, there is a finite number of solutions if  $\trace(L)^{2}-4\det(L)<0$, which is equivalent to $L$ having no real  eigenvalues.
\end{itemize}
				
			\end{proof}

			\begin{remark} In the particular case when $\gcd(\trace(L),\det(L))=1$, we can simplify the proof noting that \eqref{eq5generalcase}, \eqref{eq6generalcase} imply the existence of two sequences $(k_{n}^{1})_{n>0}, (k_{n}^{2})_{n>0}$ in $\Z$ such that $\det(L)^{n}k_{n}^{1}=\overline{m_{12}}s(n)-\overline{m_{22}}q(n)$ and $\det(L)^{n}k_{n}^{2}=-\overline{m_{12}}r(n)+\overline{m_{22}}p(n)$, i.e.,
					\begin{equation}
						\begin{array}{cl}
							k_{n}^{1} & =\overline{m_{12}}\dfrac{s(n)}{\det(L)^{n}}-\overline{m_{22}}\dfrac{q(n)}{\det(L)^{n}}\\
							&\\
							k_{n}^{2} & =-\overline{m_{12}}\dfrac{r(n)}{\det(L)^{n}}+\overline{m_{22}}\dfrac{p(n)}{\det(L)^{n}}.
						\end{array}
					\end{equation}
					
					Since $L$ is an expansion matrix, then $L^{-1}$ is a contraction, so we have that $$\lim\limits_{n\to \infty}\dfrac{p(n)}{\det(L)^{n}}=\lim\limits_{n\to \infty}\dfrac{q(n)}{\det(L)^{n}}=\lim\limits_{n\to \infty}\dfrac{r(n)}{\det(L)^{n}}=\lim\limits_{n\to \infty}\dfrac{s(n)}{\det(L)^{n}}=0,$$
					
					\noindent this implies that for all $n\in\NN$ large enough, $k_{n}^{1}=k_{n}^{2}=0$, and we conclude that $\overline{m_{12}}=\overline{m_{22}}=0$.
			\end{remark}

\cref{GeneralTheoremDifferentCasesNormalizer} implies that the linear representation group of constant-base $\Z^{2}$-odometer systems is computable. However, the techniques developed in this article may not be directly applicable to higher dimensions. This raises the following question:

\begin{question}\label{ques:LinearRepOdo} With respect to the linear representation group of higher-dimensional constant-base odometer systems, are its elements computable? Is its group structure computable?
\end{question}
By ``computable elements", we mean whether there exists an algorithm to decide whether a matrix $M$ belongs to the linear representation group or not. The second question asks if we can obtain an algorithm to determine the linear representation group, up to isomorphism,  as a function of the base matrix $L$.

\section{Minimal subshifts with infinite linear representation group}\label{sec:NormalizerSubshiftEx}

In this section, we present minimal substitutive subshifts with infinite linear representation groups, thereby providing a positive response to a question posed in \cite{baake2021number}. Their normalizer groups are fully explained. We prove the following result.
\begin{theorem}\label{prop:MainresultSection4}
For any  expansion matrix  $L \in {\mathcal M} (d, \Z) $ with $d \ge 1$, $|\det L|\geq 3$, there exists an aperiodic minimal substitutive $\Z^d$-subshift $X$ with expansion matrix $L$ such that 
\begin{enumerate}
    \item\label{it:1} It is coalescent,
    \item\label{it:2} it is an almost 1-to-1 extension of $\overleftarrow{\Z^{d}}_{(L^{n})_n}$,
    \item\label{it:3} its automorphisms are reduced to the shift transformations.
    \item\label{it:4} its linear representation semigroup $\vec{N}(X, S)$ is equal to
    \begin{align}\label{eq:NL}
        \left\{ M \in \GL: \exists n_0, \forall n,p\ge n_0 \begin{array}{l}
             L^{-n}ML^n \in \GL \\
             L^{-n}ML^n = L^{-p}ML^p \pmod {L(\Z^d)}
        \end{array} \right\},
    \end{align}
    \item\label{it:5} its normalizer group is a semidirect product of $\Z^d$ with $\vec{N}(X, S)$.  
\end{enumerate}
\end{theorem}
In the proof, we will spell out  the  isomorphisms involved  in Theorem \ref{prop:MainresultSection4}.

The  properties defining  an  element  $M$ of \eqref{eq:NL} are: $M$ belongs to the subgroup $\bigcup_{k\ge 0}\bigcap_{n\ge k} L^{n} \GL L^{-n}$, that is for each large enough integer  $n>0$, there is a $\Z^d$-automorphism $M_n$ such that
$L^n M_n  = M L^n$. Moreover, each automorphism $M_n$ permutes the  $L(\Z^{d})$-cosets. These permutations are ultimately all the same, for large enough $n$. 

Notice that in particular, when $L$ is proportional to the identity,  Theorem \ref{prop:MainresultSection4} provides an example of a minimal subshift with a linear representation semigroup equal to $\GL$.

To describe these explicit examples, we will briefly introduce some notions coming from (aperiodic) tiling theory. Most of the references come from \cite{baake2013aperiodic}.  
From a tiling perspective, the \emph{half-hex inflation} is a well-known inflation rule analogous to a symbolic substitution (for more properties about this tiling substitution, see \cite[Section 6.4]{baake2013aperiodic}). The tiles consist of 6 regular half-hexagons,  each one being the image under a 6-fold rotation of a single one. The inflation rule, up to rotation, is described in \cref{halfhextiling}. 

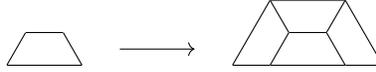
\begin{figure}[H]
	\begin{center}
		\begin{tikzpicture}[scale=0.5]
			\node(s1) at (0,0)[scale=.02]{};
			\node(s2) at (2,0)[scale=.02]{};
			\node(s3) at (1.5,0.87)[scale=.02]{};
			\node(s4) at (0.5,0.87)[scale=.02]{};
			
			\path[thin] (s1) edge (s2);
			\path[thin] (s2) edge (s3);
			\path[thin] (s3) edge (s4);
			\path[thin] (s4) edge (s1);	
			
			\node(s5) at (3,0.43)[scale=.02]{};
			\node(s6) at (5,0.43)[scale=.02]{};
			
			\path[thin,->] (s5) edge (s6);
			
			\node(s7) at (6,0)[scale=.02]{};
			\node(s8) at (7,0)[scale=.02]{};
			\node(s9) at (9,0)[scale=.02]{};
			\node(s10) at (10,0)[scale=.02]{};
			\node(s11) at (9,1.73)[scale=.02]{};
			\node(s12) at (8.5,0.87)[scale=.02]{};
			\node(s13) at (7.5,0.87)[scale=.02]{};
			\node(s14) at (7,1.73)[scale=.02]{};
			
			\path[thin] (s7) edge (s8);
			\path[thin] (s8) edge (s13);
			\path[thin] (s13) edge (s14);
			\path[thin] (s14) edge (s7);
			
			\path[thin] (s8) edge (s9);
			\path[thin] (s9) edge (s12);
			\path[thin] (s12) edge (s13);
			
			\path[thin] (s9) edge (s10);
			\path[thin] (s10) edge (s11);
			\path[thin] (s11) edge (s12);
			\path[thin] (s11) edge (s14);

		\end{tikzpicture}
	\end{center}
	\caption{Tile-substitution of the half-hex tiling.}
	\label{halfhextiling}
\end{figure}

In tiling terminology, it is an \emph{edge-to-edge inflation}, which means that each inflated tile is precisely dissected into copies of the tiles, and the vertices of any tile only intersect with the vertices of the adjacent tiles. This inflation defines an aperiodic tiling of the plane (see \cite[Example 6.4]{baake2013aperiodic}). Since the largest edge of any half-hex can only meet the largest edge of the adjacent half-hexes, two half-hexes always join to form a regular hexagon via their largest edges. By applying this procedure, the half-hex tiling can be decomposed into three hexagons, each distinguished by a single diagonal line, as shown in \cref{hexagonfigures} (see \cite[Example 6.4]{baake2013aperiodic}).

\begin{figure}[H]
	\begin{center}
		\begin{tikzpicture}[scale=0.5]
			\node(s1) at (0,0)[scale=.02]{};
			\node(s2) at (0.5,-0.87)[scale=.02]{};
			\node(s3) at (1.5,-0.87)[scale=.02]{};
			\node(s4) at (2,0)[scale=.02]{};
			\node(s5) at (1.5,0.87)[scale=.02]{};
			\node(s6) at (0.5,0.87)[scale=.02]{};
			
			\path[thin] (s1) edge (s2);
			\path[thin] (s2) edge (s3);
			\path[thin] (s3) edge (s4);
			\path[thin] (s4) edge (s5);	
			\path[thin] (s5) edge (s6);
			\path[thin] (s6) edge (s1);
			\path[thin] (s1) edge (s4);

			\node(s7) at (4,0)[scale=.02]{};
			\node(s8) at (4.5,-0.87)[scale=.02]{};
			\node(s9) at (5.5,-0.87)[scale=.02]{};
			\node(s10) at (6,0)[scale=.02]{};
			\node(s11) at (5.5,0.87)[scale=.02]{};
			\node(s12) at (4.5,0.87)[scale=.02]{};
			
			\path[thin] (s7) edge (s8);
			\path[thin] (s8) edge (s9);
			\path[thin] (s9) edge (s10);
			\path[thin] (s10) edge (s11);	
			\path[thin] (s11) edge (s12);
			\path[thin] (s12) edge (s7);
			\path[thin] (s8) edge (s11);

			\node(s13) at (8,0)[scale=.02]{};
			\node(s14) at (8.5,-0.87)[scale=.02]{};
			\node(s15) at (9.5,-0.87)[scale=.02]{};
			\node(s16) at (10,0)[scale=.02]{};
			\node(s17) at (9.5,0.87)[scale=.02]{};
			\node(s18) at (8.5,0.87)[scale=.02]{};
			
			\path[thin] (s13) edge (s14);
			\path[thin] (s14) edge (s15);
			\path[thin] (s15) edge (s16);
			\path[thin] (s16) edge (s17);	
			\path[thin] (s17) edge (s18);
			\path[thin] (s18) edge (s13);
			\path[thin] (s18) edge (s15);
			
		\end{tikzpicture}
	\end{center}
	\caption{The three tiles as a new alphabet for the half-hex tiling.}
	\label{hexagonfigures}
\end{figure}
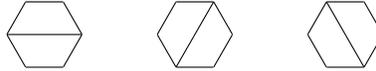

Using these full hexagons, we can define a \emph{pseudo inflation} (from \cite[Example 6.4]{baake2013aperiodic}), which is conjugated to the half-hex tiling as in \cref{PseudoInflationHalfHex}.

\begin{figure}[H]
	\begin{center}
		\begin{tikzpicture}[scale=0.5]
			\node(s1) at (0,0)[scale=.02]{};
			\node(s2) at (0.5,-0.87)[scale=.02]{};
			\node(s3) at (1.5,-0.87)[scale=.02]{};
			\node(s4) at (2,0)[scale=.02]{};
			\node(s5) at (1.5,0.87)[scale=.02]{};
			\node(s6) at (0.5,0.87)[scale=.02]{};
			
			\path[thin] (s1) edge (s2);
			\path[thin] (s2) edge (s3);
			\path[thin] (s3) edge (s4);
			\path[thin] (s4) edge (s5);	
			\path[thin] (s5) edge (s6);
			\path[thin] (s6) edge (s1);
			\path[thin] (s1) edge (s4);
			
			\node(s7) at (3,0)[scale=0.2]{};
			\node(s8) at (5,0)[scale=0.2]{};
			
			\path[thin,->] (s7) edge (s8);
			
			\node(s9) at (8,2.6)[scale=.02]{};
			\node(s10) at (9,2.6)[scale=.02]{};
			\node(s11) at (6.5,1.73)[scale=.02]{};
			\node(s12) at (7.5,1.73)[scale=.02]{};
			\node(s13) at (9.5,1.73)[scale=.02]{};
			\node(s14) at (10.5,1.73)[scale=.02]{};
			\node(s15) at (6,0.87)[scale=.02]{};
			\node(s16) at (8,0.87)[scale=.02]{};
			\node(s17) at (9,0.87)[scale=.02]{};
			\node(s18) at (11,0.87)[scale=.02]{};
			\node(s19) at (6.5,0)[scale=.02]{};
			\node(s20) at (7.5,0)[scale=.02]{};
			\node(s21) at (9.5,0)[scale=.02]{};
			\node(s22) at (10.5,0)[scale=.02]{};
			\node(s23) at (6,-0.87)[scale=.02]{};
			\node(s24) at (8,-0.87)[scale=.02]{};
			\node(s25) at (9,-0.87)[scale=.02]{};
			\node(s26) at (11,-0.87)[scale=.02]{};
			\node(s27) at (6.5,-1.73)[scale=.02]{};
			\node(s28) at (7.5,-1.73)[scale=.02]{};
			\node(s29) at (9.5,-1.73)[scale=.02]{};
			\node(s30) at (10.5,-1.73)[scale=.02]{};
			\node(s31) at (8,-2.6)[scale=.02]{};
			\node(s32) at (9,-2.6)[scale=.02]{};
			
			\path[thin] (s9) edge (s10);
			\path[thin] (s10) edge (s13);
			\path[thin] (s13) edge (s14);
			\path[thin] (s14) edge (s18);	
			\path[thin] (s18) edge (s22);
			\path[thin] (s22) edge (s26);
			\path[thin] (s26) edge (s30);
			\path[thin] (s30) edge (s29);
			\path[thin] (s29) edge (s32);
			\path[thin] (s32) edge (s31);
			\path[thin] (s31) edge (s28);	
			\path[thin] (s28) edge (s27);
			\path[thin] (s27) edge (s23);
			\path[thin] (s23) edge (s19);
			\path[thin] (s19) edge (s15);
			\path[thin] (s15) edge (s11);
			\path[thin] (s11) edge (s12);
			\path[thin] (s12) edge (s9);	
			\path[thin] (s12) edge (s13);
			\path[thin] (s13) edge (s22);
			\path[thin] (s22) edge (s29);
			\path[thin] (s29) edge (s28);
			\path[thin] (s28) edge (s19);
			\path[thin] (s19) edge (s12);
			\path[thin] (s12) edge (s16);	
			\path[thin] (s16) edge (s17);
			\path[thin] (s17) edge (s13);
			\path[thin] (s17) edge (s21);
			\path[thin] (s21) edge (s22);
			\path[thin] (s21) edge (s25);
			\path[thin] (s25) edge (s29);
			\path[thin] (s24) edge (s25);	
			\path[thin] (s24) edge (s28);
			\path[thin] (s24) edge (s20);
			\path[thin] (s19) edge (s20);
			\path[thin] (s16) edge (s20);
			\path[thin] (s20) edge (s21);
			
			\filldraw[color=gray,opacity=0.5,thick](8,2.6)--(7.5,1.73)--(8,0.87)--(7.5,0)--(8,-0.87)--(9,-0.87)--(9.5,-1.73)--(10.5,-1.73)--(11,-0.87)--(10.5,0)--(11,0.87)--(10.5,1.73)--(9.5,1.73)--(9,2.6)--(8,2.6)--cycle;
			
			\node(s33)at (8.5,0)[scale=0.2]{};
			\node(s34)at (8.5,1.73)[scale=0.2]{};
			\node(s35) at (10,0.87)[scale=0.2]{};
			
			\path[thin,->,red] (s33) edge (s35);
			\path[thin,->,red] (s33) edge (s34);
			
			\node(s36) at (9.75,0.37)[scale=1,red]{${\bm u}$};
			\node(s37) at (8.7,1.1)[scale=1,red]{${\bm v}$};
			
		\end{tikzpicture}
	\end{center}
	\caption{New tile-substitution conjugate to the half-hex tiling, with a discrete 2-dimensional tranlsation-invariant subaction in $\R^{2}$.}
	\label{PseudoInflationHalfHex}
\end{figure}
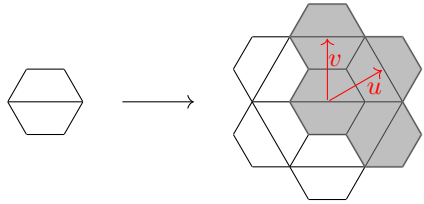

From this pseudo inflation, we construct a tiling substitution with only the four shaded hexagons in \cref{PseudoInflationHalfHex}. In this tiling substitution, there is an invariant discrete lattice $\Lambda\subseteq \R^{2}$ generated by the centers of these hexagons, using the vectors ${\bm u}$ and ${\bm v}$ as depicted in \cref{PseudoInflationHalfHex}. The discrete translation $\Lambda$-subaction is conjugate to the substitutive subshift associated with the following constant-shape substitution, called \emph{half-hex substitution}, $\zeta_{hh}$ with an expansion matrix $L_{hh}=2\cdot \Id_{\R^{2}}$ and support $F_{1}^{hh}=\{(0,0),(1,0),(0,1),(1,-1)\}$
$$\begin{array}{llllllllllllll}
	&  & 0 &  &  &  &  & 0 &  &  &  &  & 0 &  \\ 
	0 & \mapsto & 0 & 2 &  & 1 & \mapsto & 1 & 2 &  & 2 & \mapsto & 2 & 2 \\ 
	&  &  & 1 &  &  &  &  & 1 &  &  &  &  & 1. \\ 
\end{array}$$

We recall that the notation and the notions we will use are summarized in Section \ref{sec:SubstSubshift} about substitutive subshifts.
A straightforward computation, based on \cite[Theorem 4.8]{kirat2010remarksselfaffine}, reveals that the extreme points of the convex hull of $F_{n}^{hh}$  are $\{(0,0),(0,2^{n}-1),(2^{n}-1,0),(2^{n}-1,1-2^{n})\}$. Since $F_n^{hh}$ is a fundamental domain of $2^n\Z^2$, it has cardinality $4^{n}$. Furthermore, $F_{n}^{hh} \subset \conv(F_{n}^{hh})\cap \Z^{2}$, where $\conv(F_{n}^{hh})$ denotes the convex hull of $F_n^{hh}$. Actually, the Pick formula provides that the cardinality of $\conv(F_{n}^{hh})\cap \Z^{2}$ and  $F_n^{hh}$ are the same, so $F_{n}^{hh}=  \conv(F_{n}^{hh})\cap \Z^{2}$.
It then follows that $(F_{n}^{hh})_{n\geq 0}$ is a F\o lner sequence.

\begin{figure}[H]
\begin{tikzpicture}[scale=0.5]
				\draw[help lines, color=gray!30, dashed] (-0.9,-1.9) grid (1.9,2.9);
				\draw[->,thick] (-1,0)--(2,0) node[right]{$x$};
				\draw[->,thick] (0,-2)--(0,2) node[above]{$y$};
				
				\coordinate (A1) at (0,0);
				\coordinate (A2) at (0,1);
				\coordinate (A3) at (1,0);
				\coordinate (A4) at (1,-1);
				
				\draw (A4) -- (A1) -- (A2) -- (A3) -- (A4);         
				
				\fill [blue!20, fill opacity = .5]   (A4)--(A1)--(A2)--(A3)--(A4)--cycle;
				
				\foreach \v in {A1,A2,A3,A4}  \draw[fill=gray] (\v) circle (2pt);
				
				\node (b1) at (-0.5,1.2)[scale=1]{$F_{1}$};

				\draw[help lines, color=gray!30, dashed] (8.1,-2.9) grid (13.9,3.9);
				\draw[->,thick] (8,0)--(14,0) node[right]{$x$};
				\draw[->,thick] (10,-3)--(10,3) node[above]{$y$};
				
				\coordinate (A5) at (10,0);
				\coordinate (A6) at (10,1);
				\coordinate (A7) at (10,2);
				\coordinate (A8) at (10,3);
				\coordinate (A9) at (11,-1);
				\coordinate (A10) at (11,0);
				\coordinate (A11) at (11,1);
				\coordinate (A12) at (11,2);
				\coordinate (A13) at (12,-2);
				\coordinate (A14) at (12,-1);
				\coordinate (A15) at (12,0);
				\coordinate (A16) at (12,1);
				\coordinate (A17) at (13,-3);
				\coordinate (A18) at (13,-2);
				\coordinate (A19) at (13,-1);
				\coordinate (A20) at (13,0);
				
				\draw (A5) -- (A8) -- (A20) -- (A17) -- (A5);         
				
				\fill [blue!20, fill opacity = .5]   (A5) -- (A8) -- (A20) -- (A17) -- (A5)--cycle;
				
				\foreach \v in {A5,A6,A7,A8,A9,A10,A11,A12,A13,A14,A15,A16,A17,A18,A19,A20}  \draw[fill=gray] (\v) circle (2pt);
				
				\node (b2) at (8.5,1.2)[scale=1]{$F_{2}$};
			\end{tikzpicture}
			\caption{The first two supports of the half-hex substitution.}
			\end{figure}
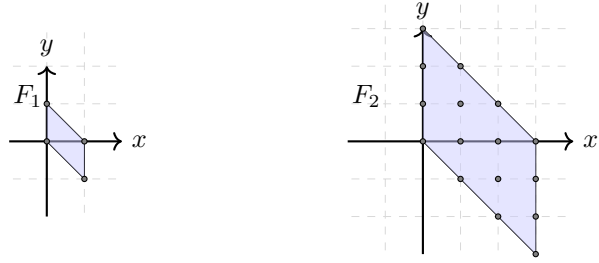

Inspired by the half-hex substitution, we consider an integer expansion matrix $L\in \mathcal{M}(d,\Z)$ with $|\det(L)|\geq 3$, a fundamental domain $F_{1}$ of $L(\Z^{d})$ in $\Z^{d}$, and set the finite alphabet $ \A= F_1\setminus\{\bm 0\}$. We define the substitution $\sigma_{L} \colon \A  \to \A^{F_1}$ as follows:
\begin{align}\label{SubStitutionToeplitzInfiniteSymmetries} \forall a \in \A,\quad \sigma_L(a)_{{\bm f}}=\left\{\begin{array}{cl}
	a & \text{ when } {\bm f}={\bm 0},\\
	{\bm f} & \text{ when } {\bm f}\neq {\bm 0}.
\end{array}\right.
\end{align}

Under the hypothesis that the sequence of supports $(F_{n})_{n>0}$ is a F\o lner sequence, we get the substitutive subshift $(X_{\sigma_{L}},S,\Z^{d})$.
It is important to notice that all the patterns $\sigma_L(a)$ coincide except at the origin, where the letter is uniquely determined.

For computational purposes, we introduce the map 
\begin{align}
    \tau \colon  \bm n \in \Z^d \setminus \{ {\bm 0}\} \mapsto  {\bm f} \in \A,
\end{align}
 where ${\bm n}= L^{p+1}({\bm z}) + L^p({\bm f})$  with ${\bm z} \in \Z^d$, ${\bm f} \in  \A$ and $p$ is the smallest integer such that  $\bm n \not\in L^{p+1}(\Z^d)$. The value $p$ serves as a multidimensional  $L$-adic valuation of $\bm n$. A motivation to introduce this map is due to the next formula that enables to compute the value of a $\sigma_L$-fixed point $\bar{x}$ at some position only by the knowledge of this position. More precisely, it is straightforward to check that
\begin{align}\label{eq:Fixedpoint}
    \forall {\bm n} \neq  {\bm 0} \in \Z^d, \quad \bar{x}_{\bm n} = \tau (\bm n).
\end{align}
This property is typical for of automatic sequences. 
As a consequence, $\sigma_{L}$ has exactly $|\A|= |\det L| - 1$ fixed points in $X_{\sigma_{L}}$, and they all coincide except at the origin. Moreover, we have the following standard  recognizability property. 

\begin{lemma}\label{lem:reconagizability} 
Let $\overline{x}$ be a fixed point of $\sigma_{L}$, then for any integer $n>0$ and any ${\bm a},{\bm b}\in \Z^{d}\setminus\{0\}$,
$$\overline{x}_{{\bm a}+F_{n}}=\overline{x}_{{\bm b}+F_{n}} \implies {\bm a}\equiv {\bm b}\ (\text{mod}\ L^{n}(\Z^{d})).$$
In particular, if the sequence of supports of the iterations $\sigma_{L}^{n}$ is  F\o lner, then the substitution $\sigma_{L}$ is recognizable on any fixed point $\overline{x}$ of ${\sigma_L}$, so $\sigma_{L}$ is aperiodic.
\end{lemma}

\begin{proof}
We prove the claim  by induction on $n>0$. We start with the base case $n=1$.

Suppose ${\bm a}\notin L(\Z^{d})$, i.e., ${\bm a}=L({\bm c})+{\bm g}$ with ${\bm g}\in F_{1}\setminus\{{\bm 0}\}$. If $b\notin L(\Z^{d})$, then ${\bm b}=L({\bm d})+{\bm h}$ with ${\bm h}\in F_{1}\setminus {\bm 0}$. Since $\overline{x}$ is a fixed point of $\sigma_{L}$, we have that $\overline{x}_{{\bm a}}=\sigma(x_{{\bm c}})_{{\bm g}}={\bm g}=\sigma(x_{{\bm d}})_{{\bm h}}$, so ${\bm g}={\bm h}$, which implies that ${\bm a}\equiv {\bm b}\ (\text{mod}\ L(\Z^{d}))$. If $b\in L(\Z^{d})$, then for any ${\bm f}\in F_{1}\setminus\{{\bm 0}\}$ we have that $\overline{x}_{{\bm b}+{\bm f}}={\bm f}=\overline{x}_{{\bm a}+{\bm f}}$. We consider ${\bm f}\in F_{1}\setminus \{{\bm 0}\}$ such that ${\bm a}+{\bm f}\notin L(\Z^{d})$, i.e., ${\bm a}+{\bm f}=L({\bm e})+{\bm h}$, so $\overline{x}_{{\bm a}+{\bm f}}={\bm h}$, and  ${\bm h}={\bm f}$, i.e., ${\bm a}\in L(\Z^{d})$ which is a contradiction.

Now, suppose there exists some $n>0$ such that $\overline{x}_{{\bm a}+F_{n}}=\overline{x}_{{\bm b}+F_{n}} \implies {\bm a}\equiv {\bm b}\ (\text{mod}\ L^{n}(\Z^{d}))$. Let  ${\bm a},{\bm b}\in \Z^{d}$ be such that
$$\overline{x}_{{\bm a}+F_{n+1}}=\overline{x}_{{\bm b}+F_{n+1}}.$$

Since $F_{n}\subseteq F_{n+1}$, by the induction hypothesis we have that ${\bm a}\equiv {\bm b}\ (\text{mod}\ L^{n}(\Z^{d}))$. We recall that $F_{n+1}=F_{n}+L^{n}(F_{1})$, so we write
$$
    {\bm a}  = L^{n+1}({\bm c}) + {\bm f}+L^{n}({\bm g}), \quad
    {\bm b}  = L^{n+1}({\bm d}) + {\bm f}+L^{n}({\bm h})
$$

\noindent for some ${\bm f}\in F_{n}$, ${\bm g},{\bm h}\in F_{1}$ and ${\bm c},{\bm d}\in \Z^{d}$. We prove that ${\bm g}={\bm h}$. If ${\bm f}={\bm 0}$ we can use a similar argument as for the case $n=1$ to conclude that ${\bm g}={\bm h}$. Suppose then ${\bm f}\neq {\bm 0}$. We consider ${\bm j}\in F_{n+1}$ such that ${\bm f}\equiv- {\bm j}\ (\text{mod}\ L^{n}(\Z^{d}))$, so
$$    {\bm a}+{\bm j}  = L^{n+1}({\bm c}_{1})+L^{n}({\bm g})\quad 
    {\bm b}+{\bm j}  = L^{n+1}({\bm d}_{1})+L^{n}({\bm h}),
$$
\noindent for some ${\bm c}_{1},{\bm d}_{1}\in \Z^{d}$.  Since $\overline{x}$ is a fixed point of $\sigma_{L}$, we get that
\begin{align*}
\overline{x}_{{\bm a}+{\bm j}} & = \sigma_{L}^{n-2}(\sigma_{L}( \sigma_{L}(\overline{x})_{{\bm c}_{1}})_{{\bm g}})_{{\bm 0}} = {\bm g}\\
\overline{x}_{{\bm b}+{\bm j}} & = \sigma_{L}^{n-2}(\sigma_{L}(\sigma_{L}(\overline{x})_{{\bm d}_{1}})_{{\bm h}})_{{\bm 0}} = {\bm h}.
\end{align*}

\noindent Recall that $\overline{x}_{{\bm a}+{\bm j}}=\overline{x}_{{\bm b}+{\bm j}}$,  hence ${\bm g}={\bm h}$ which implies that ${\bm a}\equiv {\bm b}\ (\text{mod}\ L^{n+1}(\Z^{d}))$.

\end{proof}

Recall that the map $\pi:(X_{\sigma_{L}},S,\Z^{d})\to (\overleftarrow{\Z^{d}}_{(L^{n})},{\bm +}, \Z^{d})$ is defined in Section \ref{sec:SubstSubshift}. 
\begin{proposition}\label{ParticularCaseForFibersOdometer}
	If the sequence of supports of the iterations $\sigma_{L}^{n}$ is  F\o lner, then  $\sigma_{L}$ is an aperiodic, primitive constant-shape substitution and the factor map $\pi:(X_{\sigma_{L}},S,\Z^{d})\to (\overleftarrow{\Z^{d}}_{(L^{n})},{\bm +},\Z^{d})$ is almost 1-to-1.
	
More precisely, we have  
	$$|\pi^{-1}(\{\overleftarrow{g}\})|=\left\{\begin{array}{cl}
		|\A| & \text{ if } \overleftarrow{g}\in \mathcal{O}(\overleftarrow{0},{\bm +}),\\
		1 &  \text{ otherwise.} 
	\end{array}\right.$$
\end{proposition}

In particular, the subshift $X_{\sigma_L}$ is a substitutive Toeplitz subshift and its maximal equicontinuous factor is  $\overleftarrow{\Z^{d}}_{(L^{n})}$. 
As an explicit  example, the substitutive subshift $X^{hh}$  associated with the half-hex substitution $\zeta_{hh}$ is an almost 1-to-1 extension of the constant-base odometer $\overleftarrow{\Z^{2}}_{(2^{n} \Z^{2})}$.

\begin{proof}
Since $\tau$ is a bijection, $\sigma_{L}$ is a primitive substitution. 
The aperiodicity follows from the recognizability given by \cref{lem:reconagizability}.
	
Now we study the fibers $\pi^{-1}(\{\overleftarrow{g}\})$ for $\overleftarrow{g}= (g_n)_n\in \overleftarrow{\Z^{d}}_{(L^{n})}$. Suppose $|\pi^{-1}(\{\overleftarrow{g}\})|\geq 2$ and set $x_{1},x_{2}\in \pi^{-1}(\{\overleftarrow{g}\})$, i.e., for any $n>0$ there exists $y_{1}^{(n)},y_{2}^{(n)}\in X_{\sigma_{L}}$ such that $x_{i}=S^{{g}_{n}}\sigma_{L}^{n}(y_{i}^{(n)})$, for $i\in \{1,2\}$. Let ${\bm a}\in \Z^{d}$ such that $x_{1{\bm a}}\neq x_{2{\bm a}}$. This implies that $\sigma_{L}^{n}(y_{1}^{(n)})_{{\bm a}+{g}_{n}}\neq \sigma_{L}^{n}(y_{2}^{(n)})_{{\bm a}+{g}_{n}}$. 
For every $n>0$, we write ${\bm a}+{g}_{n}=L^{n}({\bm b}_{n})+{\bm f}_{n}$, with ${\bm b}_{n}\in \Z^{d}$ and ${\bm f}_{n}\in F_{n}$. Since for $i\in \{1,2\}$ we have that $\sigma_{L}^{n}(y_{i}^{(n)})_{{\bm a}+{\bm g}_{n}}=\sigma_{L}^{n}(y_{i}^{(n)}({\bm b}_{n}))_{{\bm f}_{n}}$ and these letters are different, then for any $n>0$, ${\bm f}_{n}$ must be ${\bm 0}\in F_{n}$. This implies that for every $n>0$
	\begin{align*}
		{\bm g}_{n}\equiv -{\bm a}\ (\text{mod}\ L^{n}(\Z^{d})).
	\end{align*}
	
	Hence $\overleftarrow{g}=\kappa_{(L^{n})}(-{\bm a})$, i.e., $\overleftarrow{g}\in \mathcal{O}(\overleftarrow{0},{\bm +})$.
	It follows that $x_1$ and $x_2$ are in the orbit of two fixed points of $\sigma_L$.
	In this case, $|\pi^{-1}(\overleftarrow{g})|$ has cardinality $|\A|$ and they only differ in the coordinate $-\kappa_{(L^{n})}^{-1}(\overleftarrow{g})$. If $\overleftarrow{g}$ is not in $\mathcal{O}(\overleftarrow{0},{\bm +})$, then $\pi^{-1}(\overleftarrow{g})$ has cardinality 1. We conclude that the factor map $\pi:(X_{\sigma_{L}},S,\Z^{d})\to (\overleftarrow{\Z^{d}}_{(L^{n})},{\bm +},\Z^{d})$ is almost 1-to-1.
\end{proof}
As a consequence of  \cref{ParticularCaseForFibersOdometer}, we get the following property on the $\GL$-endomorphisms of $X_{\sigma_{L}}$.

\begin{corollary}\label{cor:FixedPoint}  
Assume the sequence of supports of the iterations $\sigma_{L}^{n}$ is  F\o lner. Then any $\GL$-endomorphism $\phi \in N(X_{\sigma_{L}},S)$ maps a $\sigma_{L}$-fixed point onto a shift of a $\sigma_{L}$-fixed point. 
\end{corollary}  
\begin{proof}
Since the factor map $\pi:X_{\sigma_{L}}\to \overleftarrow{\Z^{d}}_{(L^{n})}$ is almost 1-to-1 (\cref{ParticularCaseForFibersOdometer}), then the odometer system $\overleftarrow{\Z^{d}}_{(L^{n})}$ is the maximal equicontinuous factor of the substitutive subshift $(X_{\sigma_{L}},S,\Z^{d})$. So there exists a semigroup homomorphism $\hat{\pi}:N(X_{\sigma_{L}},S)\to N(\overleftarrow{\Z^{d}}_{(L^{n})})$ which is injective (\cref{CompatibilityPropertyFactors}). 
Recall that any equicontinuous system is coalescent, so any endomorphism is invertible and by \cref{CompatibilityPropertyEnumerateiii} of \cref{CompatibilityPropertyFactors},
any endomorphism $\phi$ satisfies

\begin{align*}
\left\{\overleftarrow{g}\in \overleftarrow{\Z^{d}}_{(L^{n})}\colon |\pi^{-1}(\{\overleftarrow{g}\})|=|\A|\right\}\subseteq \hat{\pi}(\phi)\left(\left\{\overleftarrow{g}\in \overleftarrow{\Z^{d}}_{(L^{n})}\colon |\pi^{-1}(\{\overleftarrow{g}\})|=|\A|\right\}\right).
\end{align*}
	
\noindent Since $\left\{\overleftarrow{g}\in \overleftarrow{\Z^{d}}_{(L^{n})}\colon |\pi^{-1}(\{\overleftarrow{g}\})|=|\A|\right\}$ is the orbit $\mathcal{O}(\overleftarrow{0},{\bm +}) =\kappa_{(L^{n})}(\Z^{d})$, 
it implies that $\phi$ maps the $\pi$-fiber of  the  orbit $\mathcal{O}(\overleftarrow{0},{\bm +})$ onto itself. This $\pi$-fiber consists of the orbits of $\sigma_L$-fixed points. It follows that, up to composing $\phi$ with a shift map, the image of a $\sigma_L$-fixed point $\bar{x}$ by $\phi$ is also a $\sigma_L$-fixed point. 
\end{proof}

We also characterize the $\GL$-endomorphisms of the substitutive subshift $(X_{\sigma_{L}},S,\Z^{d})$ by the following results. 
The first one concerns endomorphisms and automorphisms.

\begin{lemma}\label{LemmaNormalizerToeplitzSubstitution}
	Let $\sigma_{L}$ defined as \eqref{SubStitutionToeplitzInfiniteSymmetries} and assume the sequence of supports of the iterations $\sigma_{L}^{n}$ is  F\o lner. Then 
	\begin{itemize} 
	    \item $(X_{\sigma_{L}},S,\Z^{d})$ is coalescent, 
	    \item the automorphism group $\Aut(X_{\sigma_{L}},S)$ is trivial, i.e., consist only on the shifted transformations $S^{\bm n}$, ${\bm n} \in \Z^d$.  
	 \end{itemize} 
\end{lemma}
\begin{proof}
First we prove that $\End(X_{\sigma_{L}},S)=\left\langle S\right\rangle$. We keep the notations of \cref{CompatibilityPropertyFactors}. Set $\phi \in \End(X_{\sigma_{L}},S)$. According to \cref{cor:FixedPoint}, $\hat{\pi}(\phi)$ is an endomorphism of the odometer, which means it is a translation, as proven in \cref{lem:DescriptAutEquicont}. Moreover, since it preserves the  $\overleftarrow{0}$-orbit, $\hat{\pi}(\phi)$ is a translation by some element $\kappa_{(L^{n})}({\bm n})$ with ${\bm n}\in \Z^d$. By definition of $\hat{\pi}$, this translation is $\hat{\pi} (S^{\bm n})$, thus equal to $ \hat{\pi}(\phi)$.  We conclude, by the injectivity of $\hat{\pi}$, that $\End(X_{\sigma_{L}},S)=\left\langle S\right\rangle$. As a consequence, $(X_{\sigma_{L}},S,\Z^{d})$ is a coalescent system.
\end{proof}
	
Now, we characterize the elements of their linear representation semigroups. For this, we introduce the following notations. 	Let $N_{L}$ be the group defined by equation \eqref{eq:NL}, that is 
\begin{align*}
        N_L= \left\{ M \in \GL: \exists n_0, \forall n,p\ge n_0 \begin{array}{l}
             L^{-n}ML^n \in \GL \\
             L^{-n}ML^n = L^{-p}ML^p \pmod {L(\Z^d)}
        \end{array} \right\}.
    \end{align*}

The crucial properties of an  element  $M$ of $N_{L}$ are the following: for each large enough integer  $n>0$, there is a $\Z^d$-automorphism $M_n$ such that 
$L^n M_n  = M L^n$. Moreover, each automorphism $M_n$ permutes the  $L(\Z^{d})$-cosets. With an abuse of notation, we will denote these permutations on $F_1 \setminus\{ \bm 0\}$ by $M_n \ (\text{mod } L(\Z^d))$.  These permutations are ultimately all the same, for large enough $n$. 
Linked with the computation of the digits of  fixed points, we have the following relation 
\begin{align}\label{eq:CommutationTauM}
    \tau \circ M (\bm n)  = M_{p} \circ \tau (\bm n)  \ (\text{mod } L(\Z^d)),  
\end{align}
where $p$ is the smallest integer such that $\bm n \not\in L^{p+1}(\Z^d)$.  
	
\begin{lemma}\label{Lem:LinearRepresentation}
Assume that the sequence of supports of the iterations $\sigma_{L}^{n}$ is  F\o lner. Then the linear representation  semigroup $\vec{N}(X_{\sigma_{L}},S)$ is  the linear  group $N_{L}$.
\end{lemma}	
\begin{proof} We start by showing that $\vec{N}(X_{\sigma_{L}},S) \leqslant N_{L}$. Set $M\in \vec{N}(X_{\sigma_{L}},S)$, and let $\phi\in N(X_{\sigma_{L}}, S)$ be an $M$-endomorphism  with radius $r(\phi)$. Up to compose  $\phi$  with a shift, we can assume that it preserves the set of  $\sigma_{L}$-fixed points (\cref{cor:FixedPoint}). 

Since $\pi$ is compatible with $\GL$-endomorphisms (\cref{MaximalEquicontinuousFactorCompatibleWithHomomorphisms}), \cref{CompatibilityPropertyFactors} provides that $M\in \vec{N}(\overleftarrow{\Z^{d}}_{(L^{n})})$. This set is a group (\cref{cor:CorollariesNormalizerConditionOdometer1}), so $M^{-1}$ also belongs to  $\vec{N}(\overleftarrow{\Z^{d}}_{(L^{n})})$, i.e., for any $n>0$, there exists $m>0$ such that $L^{-n}M^{-1}L^{m}$ is an endomorphism of $\Z^{d}$ (see \cref{sec:HomoZ2Odometers}). We define $m(n)=\min\limits\{m>0: L^{-n}M^{-1}L^{m} \text{\ is an endomorphism of}\ \Z^{d}\}$.
Since the determinant of a $\Z^d$-endomorphism is an integer, we have that $m(n)\geq n$. 

We will  show that actually $m(n)=n $ for any large enough $n$, so that $M^{-1}$ belongs to $ \bigcup_{k> 0}\bigcap_{n \geq k} L^n \GL L^{-n}$. Since this set is stable by taking the inverse, this is also true for $M$. 
 
We prove the claim by contradiction, i.e., we assume there is an infinite set of integers $j$  such that $m(j)>j$. Choose $n$ large enough integer so that the ball of radius $r(\phi)$  centered at the origin is included in $L^{n}(K_{\sigma_{L}}) + F_n$, see  \cref{FiniteSubsetFillsZd}. For such $n$, $L^j(\Z^d) \cap (L^{n+1}(K_{\sigma_{L}}) + F_n)= \{\bm 0\}$, for any $j$ large enough. Since the sequence $(m(j))_{j\in S}$ goes to infinity, one can moreover assume that $m(j+1)>m(j)$.
With this convention, the  group $L^{-j}M^{-1}L^{m(j)}(\Z^d)$ can not be a subset of  $L(\Z^d)$, since otherwise this implies $m(j) \ge m(j+1)$. So there is some $\bar{\bm g} \in \Z^d$ such that  $L^{-j}M^{-1}L^{m(j)}(\bar{\bm g} ) \neq 0  \ (\text{mod } L(\Z^d))$. 
Moreover, the set $L^{-m(j)-1}ML^{j+1}(\Z^d) \setminus \Z^d $ is not empty, since otherwise the matrix $L^{-m(j)-1}ML^{j+1}$ would have integer coefficients, which is impossible since its determinant $ \det L^{j-m(j)}$ is not an integer. This provides an element $\bm h_0\in L^{-m(j)-1}ML^{j+1}(\Z^d) \setminus \Z^d$. Set $\bm h_1= L(\bm h_0 )$, we have then $\bm h_1 \not\in L(\Z^d)$ and $L^{m(j)} (\bm h_1) = M L^{j+1}(\bm h_2)$ for some element $\bm h_2 \in \Z^d$.   
These elements will enable us to provide a contradiction. 

Set $ \bm g_1 = L^{m(j)}(\bar{\bm g} )$ and  $\bm g_2 = \bm g_1 + ML^{j+1}\bm h_2 $. By construction, such elements satisfy
\begin{align*}
M^{-1} \bm g_1= L^j (L^{-j}M^{-1} L^{m(j)}(\bar{\bm g} )) = M^{-1} \bm g_2  \ (\text{mod } L^{j}(\Z^d)),
\end{align*}
so that 
\begin{align*}
\tau(M^{-1} \bm g_1) = \tau(M^{-1} \bm g_2) =  L^{-j}M^{-1} L^{m(j)} \bar{\bm g} \ (\text{mod } L(\Z^d)).
\end{align*}
Moreover the same relation and the very choice of $j$ give for any $\bm k \in K_{\sigma_L}$, $\bm f \in F_n$
\begin{align}\label{eq:g1g2}
\tau(M^{-1} \bm g_1 + L^{n+1}(\bm k) + \bm f) = \tau(M^{-1} \bm g_2 + L^{n+1}(\bm k) + \bm f).
\end{align}

On the other hand, the very choice of $\bm h_2$ implies that
$$\tau(\bm g_2) = \tau (\bar{\bm g} + \bm h_1) \qquad \text{ whereas  }  \tau(\bm g_1) = \tau (\bar{\bm g}). $$
Since $\bm h_1 \not\in L(\Z^d)$, we get 
\begin{align}\label{eq:taug1g2}
 \tau(\bm g_1) \neq \tau (\bm g_2). 
\end{align}

Consider $\overline{x}\in X_{\sigma_{L}}$ a fixed point of $\sigma_{L}$.  The relation \eqref{eq:g1g2} implies that
 $$ \bar{x}_{|M^{-1}\bm g_1 + L^{n+1}(K_{\sigma_{L}}) + F_n } = \bar{x}_{|M^{-1}\bm g_2 + L^{n+1}(K) + F_n }.$$
By the choice of $n\in\NN$ and the Curtis-Hedlund-Lyndon theorem (\cref{thm:CHLEpimorphism}), we also have that $\phi(\bar{x})_{\bm g_1} =  \phi(\bar{x})_{\bm g_2}$. 
Since $\phi$ preserves the set of $\sigma_L$-fixed points,  
we get  by \eqref{eq:Fixedpoint}, $\phi(\bar{x})_{\bm g_1}= \tau (\bm g_1)$ and  $\phi(\bar{x})_{\bm g_2}= \tau (\bm g_2)$, contradicting \eqref{eq:taug1g2}. So  $m(n)= n$ for any large enough $n$.

We still have to show that $L^{-p}ML^{p} \ (\text{mod } L(\Z^{d}))$ is uniform in $p$ for any large enough integer $p$.    
Let $K_{\sigma_{L}}$ be the finite set provided by \cref{FiniteSubsetFillsZd} and let $n$ be large enough so that  $ L^n(K_{\sigma_{L}}) + F_n$ contains the ball  $B_{r(\phi)}({\bm 0})$.
Consider $\bm f \in F_1\setminus\{0\}$ and integers $p, q > n$. We claim that  
\begin{align}\label{eq:repetition} \bar{x}_{|L^p({\bm f}) + L^n(K_{\sigma_{L}}) + F_n } =\bar{x}_{|L^q({\bm f}) + L^n(K_{\sigma_{L}}) + F_n }.
\end{align}
Indeed, by the equality \eqref{eq:Fixedpoint} and since $p>n$  we get the following for any $\bm k \in K_{\sigma_{L}}$, $\bm f_n \in F_n$
\begin{align*}
 \bar{x}_{L^p({\bm f}) + L^n(k) + \bm f_n} =   \begin{cases} 
 \tau (\bm f_n) & \text{ if }   \bm f_n \in  F_n\setminus\{0\} \\
 \tau(\bm k) &\text{ if }   \bm k \neq \bm 0 \wedge {\bm f}_{n}={\bm 0}\\
 \tau ({\bm f}) &\text{ otherwise}.
\end{cases} 
\end{align*}
In particular, notice that $\bar{x}_{|L^p({\bm f}) + L^n(K_{\sigma_{L}}) + F_n }$ is independent of $p$ and so the equality \eqref{eq:repetition} follows. 

Moreover, the equality \eqref{eq:repetition}  implies, by Curtis-Hedlund-Lyndon theorem (\cref{thm:CHLEpimorphism}), that  $\phi(\bar{x})_{ML^p{\bm f}} =  \phi(\bar{x})_{ML^q{\bm f}}$. Recall that  $M_{k}$ denotes $L^{-k} M L^{k}$ for any  integer $k \ge 0$. Since  $\phi(\bar{x})$ is also fixed by $\sigma_L$, the same computation as before provides $\phi(\bar{x})_{ML^p{\bm f}} = \tau (ML^p {{\bm f}}) = M_p \bm f \ (\text{mod } L(\Z^d))$, by \eqref{eq:CommutationTauM}. Similarly, we also have that $\phi(\bar{x})_{ML^q{\bm f}} = M_q \bm f \ (\text{mod } L(\Z^d))$. Hence  $M_p \bm f = M_q \bm f (\text{mod } L(\Z^d))$ for any $p,q>n$ and $\bm f \in F_1\setminus\{0\}$. 
It follows that $M$ is in $N_{L}$.

\medskip 

We will show now the converse inclusion, that is  $N_{L} \leqslant \vec{N}(X_{\sigma_{L}},S)$, so that the two sets are actually equal. 
 
Recall that for a matrix $M \in N_{L}$, all the matrices $M_p = L^{-p}ML^p$   permute $L(\Z^d)$-cosets, so they define an isomorphism on $F_1\setminus \{ \bm 0\} $. Moreover, this isomorphism is independent of $p$, for any $p$  greater than some $n_0\in\NN$. 
The recognizability property of $\sigma_L$ enables us to define the truncation of a ``$L$-adic" valuation as a local map. More precisely, from \cref{lem:reconagizability}, we can define the local map $v \colon {\mathcal L}_{F_{n_0}}(X_{\sigma_L}) \to \{0, 1, \ldots, n_0\}$ by 
\begin{align*}
    v(\bar{x}_{|\bm n + F_{n_0} }) =
    \begin{cases}
    q &\text{ if } \bm n \in L^{q}(\Z^d) \setminus L^{q+1}(\Z^d) \text{ with } 1 \le q \le n_0 \\ 
    n_0 & \text{ otherwise}.
    \end{cases}
\end{align*}

We then set  the map $\phi_{M}:X_{\sigma_{L}}\to\phi(X_{\sigma_{L}})$ induced by the local map 
\begin{equation}\label{eq:equationforphiM}
    \phi_{M}(x)_{{\bm n}}= M_{v(x_{|\bm M^{-1} \bm n + F_{n_0} })}  x_{M^{-1}{\bm n}} \ (\text{mod } L(\Z^d)),\quad   \text{for any } x\in X_{\sigma_{L}}, {\bm n}\in \Z^{d}.
\end{equation}

Notice that $\phi_M$ is an $M$-epimorphism onto the  subshift $\phi_M(X_{\sigma_L})$ by  Curtis-Hedlund-Lyndon theorem (\cref{thm:CHLEpimorphism}). 
Actually, the two subshifts are the same $\phi_M(X_{\sigma_L}) =X_{\sigma_{L}}$ so that $\phi_M$ is an $M$-endomorphism.   
To prove it, it is enough to show that $\phi_M$ maps a $\sigma_{L}$-fixed point $\overline{x}\in X_{\sigma_{L}}$ to another fixed point of $\sigma_L$ within $X_{\sigma_L}$, so that 
$\phi_M(X_{\sigma_L}) \cap X_{\sigma_{L}} \neq \emptyset$. The minimality of the subshift  $X_{\sigma_{L}}$ enables us to conclude that $\phi_M$ is an $M$-endomorphism. 

Indeed, Equation \eqref{eq:Fixedpoint} provides for any ${\bm n} \neq {\bm 0}  \in \Z^d$ that
\begin{align*}
  \phi_M(\bar x)_{M\bm n} &= M_{v(\bar x_{|\bm n + F_{n_0} })}  \bar x_{{\bm n}} \ (\text{mod } L(\Z^d)) \\
   &= M_{v(\bar x_{|\bm n + F_{n_0} })} \tau(\bm n)   \ (\text{mod } L(\Z^d)) \\
   &= \tau (M \bm n) \hfill \text{ by relation } \eqref{eq:CommutationTauM}. \\
\end{align*}
So $\phi_M(\bar x)$ is fixed by $\sigma_L$, and the claim follows, i.e., $\phi_M$ is an $M$-endomorphism of $X_{\sigma_L}$. Hence $\vec{N}(X_{\sigma_{L}},S)=N_{L}$, and it is a group. \cref{CoalescenceOfHommorphismsForCoalescentSystems} ensures then that $N(X_{\sigma_L}, S)$ is a group.  
\end{proof}

\begin{lemma}\label{Lem:semidirectProduct}
Assume the sequence of supports of the iterations $\sigma_{L}^{n}$ is  F\o lner, then the normalizer semigroup $N(X_{\sigma_{L}},S)$ is isomorphic to a semidirect product between $\Z^{d}$ and the linear group $N_{L}$ \eqref{eq:NL}. 
\end{lemma}
\begin{proof}
From preceding  \cref{Lem:LinearRepresentation} and \cref{LemmaNormalizerToeplitzSubstitution},
\cref{CoalescenceOfHommorphismsForCoalescentSystems} ensures then that $N(X_{\sigma_L}, S, \Z^{d})$ is a group. 
We have to prove that the map 
$$ M\in N_{L} \mapsto \phi_M \in \vec{N}(X_{\sigma_L}, S)$$ is a group embedding. This will show that the exact sequence  \eqref{ExactSequenceForNormalizer} splits and ${N}(X_{\sigma_L}, S, \Z^{d})$ is a semidirect product between $\Z^d$ and the linear group $N_{L}$.
To prove it is a group morphism, the only nontrivial point to check is the composition relation $ \phi_{MM'} = \phi_{M} \circ \phi_{M'} $ for any $M, M' \in N_{L}$. 
Since the maps $\phi_{MM'}$ and $\phi_{M} \circ \phi_{M'}$ have the same linear part, the closed set $\{x \in X_{\sigma_L} : \phi_{MM'}(x)  = \phi_{M} \circ \phi_{M'}(x) \} $ is $S$-invariant. By minimality, we only have to prove it is nonempty. We will show it contains any $\sigma_L$-fixed point $\bar{x}$. 
 Since in the previous part, we have shown that the fixed points are preserved under the maps $\phi_M$, we only have to check that the images under the two maps have the same $\bm  0$ coordinate.

Let $n_0 $ be the integer associated with $M$  such that the transformation $M_p$ coincides mod $L^p(\Z^d)$ for $p\ge n_0$. Define $n'_0$ similarly for $M'$. It is direct to check that the integer $\max(n_0, n'_0)$ plays a similar role for $MM'$. 
By definition we have  $\phi_{M'}(\bar x)_{\bm 0} = M'_{n'_0} \bar{x}_{\bm 0} \ (\text{mod } L(\Z^d))$ and $\phi_{M} \circ  \phi_{M'}(\bar x)_{\bm 0} = M_{n_0}M'_{n'_0} \bar{x}_{\bm 0} \ (\text{mod } L(\Z^d))$. Now, a direct computation gives that $(MM')_{\max(n_0, n'_0)} = M_{n_0}M'_{n'_0}  \ (\text{mod } L(\Z^d))$ and this shows the claim, i.e., the two images have the same ${\bm 0}$ coordinate. 

To prove that the morphism is injective, consider a matrix $M$ in its kernel, i.e., such that $\phi_M = \Id_{\R^{d}}$. Composing this relation with  the shift map $S^{\bm z} $, with $\bm z \in \Z^d$, and since $\phi_M$ is an $M$-endomorphism, we get that $S^{M\bm z} = S^{\bm z}$ for any  $\bm z \in \Z^d$. By aperiodicity of the subshift, $M$ has to be the identity matrix.
\end{proof}
\cref{prop:MainresultSection4} resumes all the former results. For completeness we provide a proof.
\begin{proof}[Proof of \cref{prop:MainresultSection4}]
 Items \eqref{it:1} and  \eqref{it:3}  are given by \cref{LemmaNormalizerToeplitzSubstitution}. Proposition \ref{ParticularCaseForFibersOdometer} proves Item  \eqref{it:2}. Finally  Item \eqref{it:4}  and Item \eqref{it:5} are established by \cref{Lem:LinearRepresentation} and \cref{Lem:semidirectProduct},  respectively. Moreover, the isomorphisms are explicit, precisely, we have 
 \begin{align*}
     N(X_{\sigma_{L}},S)=\{S^{{\bm n}}\phi_{M}: {\bm n}\in \Z^{d}, M\in N_{L}\},
 \end{align*}
 \noindent where $\phi_{M}$ is given by \eqref{eq:equationforphiM}.
\end{proof}

In the  particular case where  the expansion matrix $L$ is a multiple of the identity, \cref{prop:MainresultSection4} provides $\vec{N}(X_{\sigma_{L}},S)=\Cent_{\GL}(L)=\GL$. 
As a consequence, we get the following direct corollary:

\begin{corollary}
	The normalizer semigroup of the half-hex substitution $N(X_{hh},S)$ is a group and it is isomorphic to a semidirect product between $\Z^{2}$ and $\GLtwo$. Moreover, its automorphism group $\Aut(X_{hh},S)$ is trivial.
\end{corollary}

This implies that the half-hex substitutive subshift is a minimal subshift with an infinite linear representation group. In fact, since $\vec{N}(X_{hh},S)=\GLtwo$, its linear representation group is the largest possible. 
As an other example, consider the matrix $L_{\theLvariable}=\begin{pmatrix}2 & 0\\ 0 & 4\end{pmatrix}$. By \cref{GeneralTheoremDifferentCasesNormalizer}, we have that $\vec{N}(\overleftarrow{\Z^{2}}_{(L_{\theLvariable}^{n})})=\GLtwo$, but \cref{prop:MainresultSection4} and a standard analysis provide that $\vec{N}(X_{\sigma_{L_{\theLvariable}}},S)$ is the set of matrices $\left\{\begin{pmatrix}a & 2b\\ 0 & d\end{pmatrix}\colon a,d\in \{-1,1\}, b\in \Z\right\}$. In particular, $\vec{N}(X_{\sigma_{L_{\theLvariable}}},S)$ is virtually $\Z$: its quotient by the group generated by the matrix  $\begin{pmatrix}1& 2\\ 0 & 1\end{pmatrix}$ is finite.

It is then natural to ask what is the collection of all the groups $\vec{N}(X_{\sigma_{L}},S)$ that appear for all the matrices $L$ and in particular if any subgroup of  $\GLtwo$ can be realized like this. This question can be very difficult because it requires precise control of the combinatorics which can be difficult to manage. A more manageable way could be by realizing linear representation groups of specific odometers (see Question  \ref{ques:realizationOdometer}). With this, we may expect to get an answer of the following:
\begin{question}
Does there exist for any odometer system $(\overleftarrow{\Z^{d}}_{(Z_{n})})$ an almost 1-to-1 Toeplitz extension $(X,S,\Z^{d})$ such that $\vec{N}(X,S)= \vec{N}(\overleftarrow{\Z^{d}}_{(Z_{n})})$?
\end{question}
Together with Question  \ref{ques:realizationOdometer} this will enable us to enrich that family of examples with a large linear representation group.

In the more restrictive class of subshifts given by the finite data of a constant-shape substitution, it is natural to ask whether the elements of the normalizer are computable. There is a body of evidence indicating that their automorphisms can be described by an algorithm. But, as illustrated by the characterization in Theorem \ref{prop:MainresultSection4}, nothing is clear concerning the elements of the linear representation group. Related to Question \ref{ques:LinearRepOdo}, we ask the following:  \begin{question}\label{ques:LinearRepSub}  Regarding the linear representation group for substitutive constant-shape subshifts, are its elements computable?  Is its group structure computable?
\end{question}
Here we mean  ``computable elements" in the sense that  there is an algorithm deciding  whether or not  a matrix $M$ belongs to the linear representation group. The second question is to find an algorithm for determining the linear representation group, up to isomorphism, as a function of the substitution.

\bibliographystyle{plain}
\bibliography{sample}

\end{document}